\documentclass[a4paper,10pt]{article}
\usepackage{amsmath,amsthm,amssymb,graphicx}

\usepackage{enumerate}
\usepackage{enumitem}
\usepackage{float}
\usepackage{tikz}
\usetikzlibrary{graphs,graphs.standard,shapes.geometric,calc}
\usepackage{float}
\usepackage{url}
\usepackage{hyperref}
\usepackage{algorithm}
\usepackage[noend]{algpseudocode}
\usepackage{caption}

\renewcommand{\title}[1]{\vspace{\fill}
\eject\addtolength{\baselineskip}{4pt}
{\bfseries\LARGE #1}\\[3mm]\addtolength{\baselineskip}{-4pt}}
\renewcommand{\author}[3]{\parbox[t]{75mm}
{\begin{center}{\scshape #1}\\[3mm] #2\\
 {\ttfamily #3} \end{center}}}

\newcommand{\ml}{${\rm ml}$}

\newtheorem{thm}{\bfseries Theorem}
\newtheorem{lem}[thm]{\bfseries Lemma}        
\newtheorem{remark}[thm]{\bfseries Remark}    
\newtheorem{prop}[thm]{\bfseries Proposition} 
\newtheorem{cor}[thm]{\bfseries Corollary}     

\newtheorem{cl}[thm]{\bfseries Claim}

\begin{document}

\begin{center}

\title{Network fault costs based on\\[1mm] minimum leaf spanning trees
 } 
\author{Jan Goedgebeur\footnotemark[1]  
}{
Department of Computer Science\\
KU Leuven Kulak\\
Etienne Sabbelaan 53, 8500 Kortrijk, Belgium \\ and \\ Department of Mathematics, Computer Science and Statistics \\
Ghent University\\
Krijgslaan 281-S9, 9000 Ghent, Belgium
}{
jan.goedgebeur@kuleuven.be
}\footnotetext[1]{Research is supported by Internal Funds of KU Leuven and an FWO grant with grant number G0AGX24N.}
\author{Jarne Renders\footnotemark[1] 
}{Department of Computer Science\\
KU Leuven Kulak\\
Etienne Sabbelaan 53, 8500 Kortrijk, Belgium
}{
jarne.renders@kuleuven.be
}
\author{{G\'abor Wiener}\footnotemark[2]
}{ 
Department of Computer Science and Information Theory\\
Budapest University of Technology and Economics\\
M\H uegyetem rkp. 3., 1111 Budapest, Hungary 
}{
wiener@cs.bme.hu
}\footnotetext[2]{Research is supported by project no.\ BME-NVA-02, implemented with the support provided by the Ministry of Innovation and Technology of Hungary from the National Research, Development and Innovation Fund, financed under the TKP2021 funding scheme.}
\author{Carol T. Zamfirescu  
}{
Department of Mathematics, Computer Science and Statistics \\
Ghent University\\
Krijgslaan 281-S9, 9000 Ghent, Belgium\\ and \\ Department of Mathematics \\ Babe\c{s}-Bolyai University \\ Cluj-Napoca, Roumania
}{
czamfirescu@gmail.com
}


\end{center}


\begin{quote}
{\bfseries Abstract:}
We study the fault-tolerance of networks from both the structural and computational point of view using the minimum leaf number of the corresponding graph $G$, i.e.\ the minimum number of leaves of the spanning trees of $G$, and its vertex-deleted subgraphs. We investigate networks that are leaf-guaranteed, i.e.\ which satisfy a certain stability condition with respect to minimum leaf numbers and vertex-deletion. Next to this, our main notion is the so-called \emph{fault cost}, which is based on the number of vertices that have different degrees in minimum leaf spanning trees of the network and its vertex-deleted subgraphs. We characterise networks with vanishing fault cost via leaf-guaranteed graphs and describe, for any given network $N$, leaf-guaranteed networks containing $N$. We determine for all non-negative integers $k \le 8$ except 1 the smallest network with fault cost $k$. We also give a detailed treatment of the fault cost 1 case, prove that there are infinitely many 3-regular networks with fault cost 3, and show that for any non-negative integer $k$ there exists a network with fault cost exactly~$k$.
\end{quote}

\begin{quote}
\end{quote}
\newpage

\section{Introduction} 
 
We investigate the fault-tolerance of networks using spanning trees of the corresponding graphs. Optimisation problems concerning spanning trees occur in many applications, such as querying in computer database systems and connection routing. Consider a network modelled by a graph $G$. In various applications, in order to minimise costs, it is desirable to describe a spanning tree of $G$ with as few leaves as possible. The number of leaves in such a tree is called the minimum leaf number---the formal definition is given below. This number shall measure the quality of our solution: the smaller the number of leaves in a spanning tree, the better. In this article we want to investigate the situation when, due to a technical failure in the network, one of its nodes becomes unreachable.
We model this by removing the corresponding vertex, which we call $v$, from $G$. In order to achieve a \emph{fault-tolerant} network, we wish to guarantee that the minimum leaf number of $G - v$ is not larger than that of $G$, i.e.\ vertices in $G - v$ remain at least as well reachable as in $G$, and this holds true for an \emph{arbitrary} node of the network as we do not know where the failure will occur. We shall here formalise this idea and give both theoretical as well as computational results.
We will do so from two angles: on the one hand, we will investigate so-called leaf-guaranteed graphs, in which indeed vertices in $G - v$ remain at least as well reachable as in $G$, while on the other hand we shall introduce the notion of \textit{fault cost}, a general tool to gauge how well a given arbitrary network performs with respect to an occurring fault (i.e.,
in our model, the loss of a node).
The vertex set and the edge set of a graph $G$ is denoted by $V(G)$ and $E(G)$, respectively. The subgraph of $G$ induced by $X \subseteq V(G)$ is denoted by $G[X]$ and let $G-X := G[V(G)\setminus X]$, $G-v := G-\{ v\}$ for any $v\in V(G)$. For $v,w \in V(G)$ let $G + vw$ denote the graph obtained from $G$ by adding the edge $vw$ to $E(G)$. A graph is \emph{connected} if there is a path between any two of its vertices. For an integer $k \ge 1$, if a connected graph $G$ has a set of $k$ vertices $X$ for which $G - X$ has at least two connected components, we call $X$ a \emph{$k$-separator} of $G$. The graph $G$ is \emph{$k$-connected} if it has more than $k$ vertices and does not have an $m$-separator for any $m < k$, and $G$ is of \emph{connectivity} $k$ if it is $k$-connected and admits a $k$-separator or a set of $k$ vertices whose removal leaves a single vertex. For vertices $v$ and $w$ in a graph, we write \textit{$v$-path} for a path with end-vertex $v$, and \textit{$vw$-path} for a $v$-path with end-vertex $w$. We denote the set of all spanning trees of $G$ as ${\cal T}(G)$. A \textit{leaf} is a vertex of degree~1, and, in a tree, a \textit{branch} is a vertex of degree at least 3. A spanning tree with exactly $k$ leaves is a \textit{$k$-leaf spanning tree}. We call $L(G)$ the set of leaves of a graph $G$ and put $\ell(G) := |L(G)|$. A vertex is \textit{cubic} if it has degree~3, and a graph is \textit{cubic} or \textit{$3$-regular} if all of its vertices are cubic. For a subgraph $H$ of some given graph, if a vertex $v$ in $H$ has exactly $k$ neighbours in $H$, we say that $v$ has \textit{$H$-degree} $k$. Throughout the paper, we assume graphs to be simple, undirected, and 2-connected, unless explicitly stated otherwise.

A graph on $n$ vertices is \textit{hamiltonian} if it contains a cycle of length $n$, i.e.\ a \textit{hamiltonian cycle}, and it is \textit{traceable} if it contains a path on $n$ vertices, i.e.\ a \textit{hamiltonian path}. We do not consider $K_1$ or $K_2$ to be hamiltonian. Following~\cite{Wi17}, the \emph{minimum leaf number} ml$(G)$ of a graph $G$ is defined to be 1 if $G$ is hamiltonian, $\infty$ if $G$ is disconnected, and $\min_{T \in {\cal T}(G)} \ell(T)$ otherwise. It is worth pointing out that the minimum leaf number of a graph is bounded above by its independence number, see~\cite[Proposition~8]{GHHSV04}. In a graph $G$, we will call a spanning tree or hamiltonian cycle $S$ with ${\rm ml}(S) = {\rm ml}(G)$ an \emph{ml-subgraph}. For further results on trees with a minimum number of leaves and equivalent problems, we refer to~\cite{GHHSV04,SW08,BFGL13,GOVW19}, and for a pertinent US Patent of Demers and Downing, Oracle Corp., see~\cite{DD00}.

A graph $G$ is \emph{$k$-leaf-guaranteed} if $k = {\rm ml}(G) \ge {\rm ml}(G - v)$ for all $v \in V(G)$. We denote the family of all $k$-leaf-guaranteed graphs with ${\cal L}_k$ and call a graph $G \in \bigcup_k {\cal L}_k$ \emph{leaf-guaranteed}. It is easy to show (but we nonetheless give the proof in the next proposition, for completeness' sake) that $\{{\rm ml}(G - v)\}_{v \in V(G)} \subset \{ k - 1, k \}$ for any $k$-leaf-guaranteed graph $G$.
We write ${\cal L}_k^\ell$ for the set of all graphs $G \in {\cal L}_k$ satisfying ${\rm ml}(G - v) = \ell$ for all $v \in V(G)$, where $\ell \in \{ k - 1, k\}$. Note that ${\cal L}_k \ne {\cal L}_k^k \cup {\cal L}_k^{k-1}$, since there exist graphs of which the vertex-deleted subgraphs  have non-constant minimum leaf number. 

Deciding whether a given graph has a certain minimum leaf number is an NP-complete problem, as it includes the hamiltonian cycle problem as a special case. Leaf-guaranteed graphs offer a common framework for a series of important graph families. ${\cal L}_k^k$ and ${\cal L}_k^{k-1}$ are exactly the \emph{leaf-stable} and \emph{leaf-critical} graphs, respectively, as investigated in~\cite{Wi17}. These generalise hypohamiltonian and hypotraceable graphs and their applications include the solution to a problem of Gargano et al.~\cite{GHHSV04} concerning the wave division multiplexing technology in optical communication. We recall that a graph is \textit{hypohamiltonian} (\textit{hypotraceable}) if the graph itself is non-hamiltonan (non-traceable) yet all of its vertex-deleted subgraphs are hamiltonian (traceable). Hypohamiltonicity---for an overview of theoretical results see the survey of Holton and Sheehan~\cite[Chapter~7]{HS93}---has been applied in operations research in the context of the monotone symmetric traveling salesman problem~\cite{Gr80,GW81} as well as in coding theory~\cite{LMPC12}. Applications related to hypohamiltonicity also appear in the context of designing fault-tolerant networks, see~\cite{LKP05, WHH98}. Moreover, various graph generation algorithms have been designed for hypohamiltonian graphs and related families~\cite{AMW97,GNZ20}. The family ${\cal L}_1^1$ is known as the \emph{$1$-hamiltonian} graphs, a classical notion in hamiltonicity theory~\cite{CKL70}. In applications, these graphs are often called \emph{$1$-vertex fault-tolerant}~\cite[Chapter~12]{HL09}. ${\cal L}_2 \cup {\cal L}_{3}^2$ are exactly the so-called \emph{platypus graphs}~\cite{GNZ20,Za18} which are connected to the Steiner-Deogun property~\cite{KLM96} as described in~\cite{Z}. 


In Section~\ref{sec:lgg} we give structural properties of leaf-guaranteed graphs and then prove that for any network~$N$ on $n$ nodes one can describe a leaf-guaranteed network with fewer than $16n$ nodes which contains $N$. In Section~\ref{sec:fc}, the paper's main results are given. These revolve around the notion of a network's \textit{fault cost}. We first give this new notion's formal definition and then give a series of structural results. This is complemented by an algorithm, which we also implemented, to compute the fault cost. For instance, we determine for all non-negative integers $k \le 8$ except 1 the smallest network with fault cost~$k$. We also give a detailed treatment of the difficult fault cost 1 case, prove that there are infinitely many 3-regular networks with fault cost 3, and show that for any non-negative integer $k$ there exists a graph with fault cost exactly $k$. The paper concludes with Section~\ref{sec:probs}, in which we discuss open problems.

\section{Leaf-guaranteed graphs}\label{sec:lgg} 

We begin by summarising some fundamental properties of leaf-guaranteed graphs. 

\begin{prop}\label{prop:lgg} 
Let $G \ne K_2$ be a $k$-leaf-guaranteed graph. Then the following hold.
\begin{enumerate}[label=\normalfont(\roman*)]
\item $G$ is $2$-connected, but not necessarily $3$-connected, and the maximum degree of $G$ is at least~$3$.
\item ${\rm ml}(G - v) \in \{ k-1, k \}$ for all $v \in V(G)$.  
\item For any vertex $v$ in $G$ there exists an ml-subgraph $T$ of $G$ such that $v$ is not a leaf of $T$. Moreover, for every $k$ there exists a $k$-leaf-guaranteed graph $G$ containing a vertex $x$ which is not a leaf in any ml-subgraph of $G$. 
\end{enumerate}
\end{prop}

\begin{proof} (i) Assume $G$ has a $1$-separator $\{ x \}$. Then $G - x$ is disconnected, whence ${\rm ml}(G) = \infty$, a contradiction. Every leaf-stable graph and every leaf-critical graph is leaf-guaranteed. Leaf-stable and leaf-critical graphs of connectivity 2 were described in~\cite{OWZ20}, so leaf-guaranteed graphs need not be 3-connected. Since a leaf-guaranteed graph is 2-connected but cycles are not leaf-guaranteed, its maximum degree must be at least~3. 

(ii) Since ${\rm ml}(K_1) = 0$ and ${\rm ml}(K_2) = 2$, the assertion does not hold if $G = K_2$, although $K_2$ is leaf-guaranteed. We now prove the statement for every graph $G \ne K_2$ and assume this henceforth tacitly. We have that ${\rm ml}(G)$ is a positive integer or $\infty$ unless $G = K_1$. As $G$ is leaf-guaranteed we have ${\rm ml}(G - v) \le {\rm ml}(G)$ for any vertex $v$ in $G$ by definition, so the statement holds for all graphs $G$ with ${\rm ml}(G) \le 2$. 

So let ${\rm ml}(G) = k \ge 3$. We continue with a proof by contradiction and suppose ${\rm ml}(G - v) \le k - 2$ for some vertex $v$ in $G$. If there exists a hamiltonian cycle $\mathfrak{h}$ of $G - v$ we can easily modify $\mathfrak{h}$ to a hamiltonian path in $G$, so ${\rm ml}(G) \le 2$, contradicting ${\rm ml}(G) = k \ge 3$. So ${\rm ml}(G - v) \ge 2$ certainly holds for all $v$ in $G$.

Hence we may assume that there exists a spanning tree $T$ of $G - v$ with at most $k - 2 \ge 2$ leaves. Let $w \in N(v)$. Then $\ell(T + vw) \le k - 1$, so ${\rm ml}(G) \le k - 1$, a contradiction. On the other hand, ${\rm ml}(G - v) \ge k + 1$ is impossible, since $G$ is $k$-leaf-guaranteed, so by definition ${\rm ml}(G - v) \le k$.

(iii) Both statements are obviously true for $k = 1$ (as no leaves are present), so assume henceforth $k \ge 2$. In a $k$-leaf-guaranteed graph $G$, consider a spanning tree $T$ with $k$ leaves, one of which shall be $v$. Let $w$ be adjacent to $v$ in $T$ and $u \ne w$ a vertex adjacent to $v$ in $G$ ($u$ exists as $G$ is 2-connected by Proposition~2). Add the edge $uv$ to $T$, resulting in the graph $T'$.

If $u$ is a leaf in $T$, denote by $s$ the vertex adjacent to $u$ in $T$. Removing $us$ from $T'$ we obtain a new spanning tree $T_0$ of $G$. In $T$ as well as $T_0$ the vertex $u$ is a leaf, but $v$ is not a leaf anymore in $T_0$, so $T_0$ certainly does not have more leaves than $T$. In $T_0$ the vertex $s$ must be a leaf, as otherwise $T_0$ would be a spanning tree of $G$ with fewer than $k$ leaves, which is absurd. So $T_0$ is a spanning tree of $G$ with $\ell(T) = \ell(T_0)$ and $v$ is a leaf of $T_0$.

Consider now the situation that $u$ is not a leaf in $T$. Remove from $T'$ the unique edge lying on the unique cycle of $T'$ and incident with $u$ but not $v$. We obtain a new spanning tree $T_1$ of $G$ in which $v$ is not a leaf. Since $\ell(T) = k = {\rm ml}(G)$ the inequality $\ell(T_1) < \ell(T)$ is impossible, so $\ell(T_1) = \ell(T).$ 

For the second statement, consider the Petersen graph $P$ and two adjacent vertices $v$ and $w$ in $P$. It is well-known that both $P$ and $P - w$ contain a hamiltonian path with an endpoint in $v$, a property we will refer to as $(\star)$. (The former follows from the fact that $P$ is traceable and vertex-transitive, and the latter from the fact that $P$ is hypohamiltonian.)
Consider $k$ pairwise disjoint copies $P^1, \ldots, P^k$ of $P - vw$, denoting the respective copies of $v$ and $w$ in $P^i$ by $v_i$ and $w_i$. We construct a graph $G_k$ by identifying all $v_i$'s, yielding one vertex $x$, and identifying all $w_i$'s, yielding one vertex $y$. We shall see the $P^i$'s as subgraphs of $G_k$.

Let $T$ be a spanning tree of $G_k$. Since $P$ is non-hamiltonian, $T$ contains at least one leaf in each component of $G_k - x - y$. Therefore, ${\rm ml}(G_k) \ge k$. On the other hand, by $(\star)$, the graph $G_k$ contains a spanning tree with exactly $k$ leaves, whence, the minimum leaf number of $G_k$ is precisely $k$. In order to prove that $G_k$ is $k$-leaf-guaranteed, we need to show that ${\rm ml}(G - v) \le k$ for every $v \in V(G_k)$. We need to differentiate between two cases. If $v \in \{ x, y \}$, then we make use of $(\star)$ and obtain the desired conclusion. If $v \notin \{ x, y \}$, then there exists an $i \in \{ 1, \ldots, k \}$ such that $v \in V(P^i)$. Since $P$ is hypohamiltonian, there exists a hamiltonian $x$-path in $P^i - v$. Combining this with $(\star)$, we obtain a spanning tree of $G_k - v$ with exactly $k$ leaves. Thus, $G_k$ is $k$-leaf-guaranteed. To conclude the proof, we will show that there exists no spanning tree $T$ in $G_k$ which has $k$ leaves, and one of the leaves of $T$ is $x$. However, as noted above, $T$ contains at least one leaf in each component of $G_k - x - y$, of which there are $k$ by construction. Therefore, $T$ has at least $k + 1$ leaves, a contradiction. \end{proof}


It is natural to ask whether, given a network $N$, we can find a larger network $N'$ containing a copy of $N$ (or, in graph-theoretical terms: $N$ is an induced subgraph of $N'$) such that $N'$ is leaf-guaranteed. It is known that this is possible if $N'$ may be 40 times larger than $N$: in~\cite{ZZ18} it was shown that any graph $G$ of order at most $n$, connected or not, occurs as induced subgraph of some hypohamiltonian and thus 2-leaf-guaranteed graph of order $40n$. Relaxing the conditions and only requiring containment in a leaf-guaranteed graph, we can describe significantly smaller solutions.

\begin{thm} \label{thm:ind-subgr}
Let $G \ne K_1$ be a possibly disconnected graph of order $n$ containing a longest path~$\mathfrak{p}$ which has $p$ vertices. Put $k := 2n - p + 1$. Then $G$ occurs as an induced subgraph of some $1$-leaf-guaranteed graph of order $k + 1$ as well as of a $k$-leaf-guaranteed graph of order 
$8k$.
\end{thm}

\begin{proof} Consider a $k$-cycle $C = v_1 \ldots v_k$ disjoint from $G$. Put $\mathfrak{p} = w_1 \ldots w_p$ and $V(G) \setminus V(\mathfrak{p}) = \{ w_{p+1}, \ldots, w_n \}$. Identify $v_i$ with $w_i$ for all $i \in \{ 1, \ldots, p \}$ and $v_{p + 2j}$ with $w_{p+j}$ for all $j \in \{ 1, \ldots, n-p \}$. We have produced a graph $H$, where the purpose of the identification we just performed is to guarantee that $G$ is indeed an induced subgraph of $H$.

For the first part of the statement, consider the graph $H'$ of order $2n - p + 2 = k + 1$ obtained by taking $H$ and an additional vertex $v_0$, and joining by an edge $v_0$ to $v_i$ for all $i \in \{ 1, \ldots, k \}$. That $H'$ is 1-leaf-guaranteed is easy to see: Considering $C$ now as a subgraph of $H'$, we can replace the edge $v_1v_2$, which lies in $C$, by the path $v_1v_0v_2$ to obtain a hamiltonian cycle in $H'$. The cycle $C$ itself is a hamiltonian cycle in $H' - v_0$. Consider $v$, an arbitrary vertex in $H' - v_0$, and let $u$ and $w$ be $v$'s neighbours on $C$. Adding to the $uw$-path $C - v$ the path $uv_0w$ yields a hamiltonian cycle in $H' - v$. Finally, it is clear that $G$ is an induced subgraph of $H'$, so the first part of the theorem is proven.

Let $\Xi_8$ be the 8-vertex graph obtained by connecting two vertices by three edges, subdividing once any two of the three edges, and blowing up the two cubic vertices to triangles (see the right-hand side of Fig.~\ref{fig:replace_with_Xi_8}). For the theorem's second statement, replace in $H$ every edge of $C$ by a copy of $\Xi_8$, see Fig.~\ref{fig:replace_with_Xi_8}. We now show that the resulting graph $H''$ is $k$-leaf-guaranteed. Let $S \in {\cal S}_{{\rm ml}}(H'')$. It is easy to check that $\Xi_8$ contains no hamiltonian $vw$-path, where $v$ and $w$ are vertices as shown in Fig.~\ref{fig:replace_with_Xi_8}. Thus $S$ must contain at least one leaf in each copy of $\Xi_8$ present in $H''$, whence $\ell(S) \ge k$.

\begin{figure}[!htb]
    \centering
    \begin{tikzpicture}[fo/.style={draw, circle, fill=black, minimum size={0.15cm}, inner sep=0cm, scale=0.65}, scale=0.4]
        \node[fo] (v) at (0,0) {};
        \node[fo] (w) at (8,0) {};
        \node[] (v') at (-3,-1) {};
        \node[] (w') at (11,-1) {};
        \draw (v) -- (w);
        \draw (v) -- (v');
        \draw (w) -- (w');
    \end{tikzpicture}
    \begin{tikzpicture}[scale=0.4]
        \node[scale=1.5] at (0,2) {$\longrightarrow$};
        
        \node[fill=white] at (0,-1) {}; 
        \node[fill=white] at (0,4) {};
    \end{tikzpicture}
    \begin{tikzpicture}[fo/.style={draw, circle, fill=black, minimum size={0.15cm}, inner sep=0cm, scale=0.65}, scale=0.4]
        \node[fo, label=below:$v$] (v) at (0,0) {};
        \node[fo] (1) at (0,4) {};
        \node[fo] (2) at (2,2) {};
        \node[fo, label=below:$w$] (w) at (8,0) {};
        \node[fo] (3) at (8,4) {};
        \node[fo] (4) at (6,2) {};
        \node[fo] (5) at (4,4) {};
        \node[fo] (6) at (4,2) {};
        \node[] (v') at (-3,-1) {};
        \node[] (w') at (11,-1) {};
        \draw (v) -- (w);
        \draw (v) -- (v');
        \draw (w) -- (w');
        \draw (v) -- (1) -- (2) -- (v);
        \draw (w) -- (3) -- (4) -- (w);
        \draw (1) -- (5) -- (3);
        \draw (2) -- (6) -- (4);
    \end{tikzpicture}
    \caption{Replacing an edge with a copy of the graph $\Xi_8$.}
    \label{fig:replace_with_Xi_8}
\end{figure}
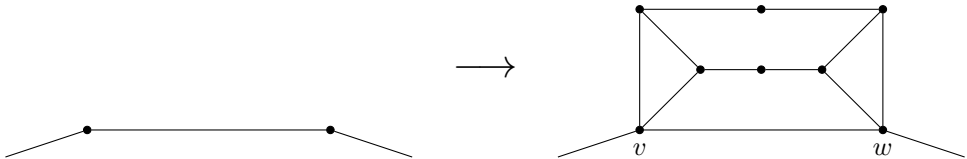

It is straightforward to construct a tree with exactly $k$ leaves in $H''$, so $\mathrm{ml}(H'') = k$. It remains to prove that $\mathrm{ml}(H'' - v)\le k$ for every $v \in V(H'')$. By construction, we can restrict ourselves to an arbitrary but fixed copy of $\Xi_8$ in $H''$. It is easy to check that for every vertex $u$ in $\Xi_8$, there exists a hamiltonian path in $\Xi_8 - u$ with at least one endpoint in $V(\Xi_8)\cap V(C)$; see Fig.~\ref{fig:Xi_8_K1-traceable} in Appendix~\ref{app:figure_proof_thm2}. As before, we can complete this path to a spanning tree of $H'' - u$ with exactly $2k$ leaves. We have proven that ${\rm ml}(H'' - v) \le k$ for every vertex $v$ in $H''$, whence, $H''$ is $k$-leaf-guaranteed.
\end{proof}

We end this section with a brief remark regarding a question of Gr\"otschel. He asked in~\cite[Problem 4.56]{GGL95} whether \textit{bipartite} hypotraceable graphs exist, i.e.\ whether any member of ${\cal L}_3^2$ is bipartite; note that hypohamiltonian graphs, i.e.\ members of ${\cal L}_2^1$, cannot be bipartite. As there has been little progress on this question, we relax it and ask for bipartite leaf-guaranteed graphs. 
In the upcoming Section~\ref{sec:computation} we describe an algorithm to determine the fault cost of a graph, a notion we shall define later. For a given graph $G$ it computes ml-subgraphs of $G$ and $G - v$ for all $v\in V(G)$. A straightforward adaptation allows to verify whether a graph is leaf-guaranteed. We generate $2$-connected bipartite graphs using \texttt{geng}~\cite{MP14} and then apply our program~\cite{GRWZ25} to check whether the input graphs are $k$-leaf-guaranteed for a certain $k$.
Our computations yield that up to order $15$, there is exactly one $2$-connected bipartite leaf-guaranteed graph
of order $12$ (see Fig.~\ref{fig:bip_leaf_guar}) and 
five more $2$-connected bipartite leaf-guaranteed graphs on order $14$. Restricting to girth at least $6$ up to order 22 there is this aforementioned graph on order $12$ as well as six more leaf-guaranteed graphs on order 16; $27$ on order 18; $815$ on order 20; and 11\,775 on order 22.
All of them are $2$-leaf-guaranteed bipartite graphs and hence are $2$-leaf-stable, i.e.\ in $\mathcal{L}_2^2$. We did not find any $k$-leaf-guaranteed examples with $k > 2$.


\begin{figure}[!htb]
    \centering
    \newcommand{\s}{8}
    \begin{tikzpicture}[transform shape, scale=1.5, rotate=-45]
        \tikzstyle{fo} = [draw, circle, fill=black, minimum size={0.15cm}, inner sep=0cm, scale=0.65]
        
        \node[draw=black,minimum size=2cm,regular polygon,regular polygon sides=\s] (a) {};
        
        \foreach \x in {1,2,...,\s}
        	\node[fo] (a\x) at (a.corner \x) {}; 
        \foreach \x in {1,2,...,\the\numexpr\s/2\relax} {
        	\draw (a\x) to (a\the\numexpr\x+\s/2\relax); 
            \node (b) at (a\the\numexpr\x+\s/2\relax) {};
            \node[fo] at ($(a\x)!0.3!(b)$) {};
        }
    \end{tikzpicture}
    \caption{The smallest bipartite leaf-guaranteed graph.}
    \label{fig:bip_leaf_guar}
\end{figure}
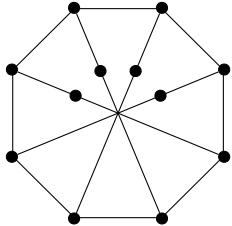

\section{Fault cost}\label{sec:fc}

\subsection{Definition}

We now introduce the notion of fault cost. From an application-oriented perspective it is important to point out that when a node drops from the network, the ml-subgraph used in the fault-free network may require changing equipment\footnote{We see vertices of different degrees as requiring different equipment in the network. As Salamon and Wiener write in~\cite{SW08}, various problems related to ours have an objective function that depends on the vertex degrees of the spanning tree, see for instance~\cite{GHHSV04,LR98,So98}. This model is particularly useful when designing networks where device costs depend on the required routing functionality.} in many nodes in order to obtain an ml-subgraph in the faulty network, which is undesirable. We formalise this by introducing, for a spanning tree or hamiltonian cycle $S$ of $G$ and a spanning tree or hamiltonian cycle $S_v$ of $G - v$, the \emph{transition cost from $S$ to $S_v$} as
$$\tau(S,S_v) := |\{ w \in V(G) \setminus \{ v \} : \deg_S(w) \ne \deg_{S_v}(w) \}|.$$
Thus, this is the number of vertices in $G$ which need to receive different equipment after the loss of a node; the lost node is ignored in this process. For a given graph $G$, denote by ${\cal S}_{{\rm ml}}(G)$ the set of all of its ml-subgraphs. In order to quantify the optimal solution in a worst-case scenario for the network itself, we introduce for a given graph $G$ and an ml-subgraph $S$ of $G$ the quantity
$$\varphi_S(G) := \max_{v \in V(G)} \min_{S_v \in {\cal S}_{{\rm ml}}(G - v)} \tau(S,S_v).$$
Based on this, we shall consider the \emph{fault cost} $\varphi(G)$ of the graph $G$ representing the network:
$$\varphi(G) := \min_{S \in {\cal S}_{{\rm ml}}(G)} \varphi_S(G).$$

We give an example of how these definitions work in Appendix~\ref{app:fc2_example}. This framework is motivated by the fact that in a practical application one knows the network but must choose one of the ml-subgraphs without knowing which node might
fail---so a subgraph ought to be chosen which minimises the maximum transition cost. We
shall call such a subgraph, i.e.\ an ml-subgraph $S$ in $G$ satisfying $\varphi_{S}(G) = \varphi(G)$, \textit{optimal}.
Finally, when a vertex does fail, one again has the freedom to choose an ml-subgraph (in the
faulty network) bearing the lowest transition costs\footnote{In this model, it may happen that for a certain ml-subgraph only a few nodes are responsible for a high transition cost (and all other nodes yield far lower transition costs), but that for another ml-subgraph the transition costs are more evenly distributed and on average worse, but that the worst-case transition cost is in fact lower, so this latter tree will be chosen. This may not be desirable in certain applications, in particular when nodes fail with varying probabilities. However, in this article we are interested in conserving (or decreasing) the minimum leaf number whenever \emph{any} vertex fails, irrespective of failure probability.}.

\subsection{Computational approach}\label{sec:computation}


\subsubsection{An algorithm for computing the fault cost}
\label{subsect:algo}

Henceforth we focus on the description, both structural as well as computational, of graphs with small fault cost. For the latter approach we designed an algorithm to computationally determine the fault cost of a given graph.
We use a backtracking approach and efficiently compute and store the distinct degree sequences of all ml-subgraphs and its vertex-deleted subgraphs.
Once they have been determined, the degree sequences are compared and the fault cost determined. We use pruning rules, explained in more detail later in this section, in several places to speed up the computations. Our implementation of the algorithm is open source and can be found on GitHub~\cite{GRWZ25}. An overview of the algorithm can be found in Algorithm~\ref{alg:fc}. 

\begin{algorithm}[!htb]
\caption{computeFaultCost($G$)}\label{alg:fc}
\begin{algorithmic}
    \If{$G$ is $1$-hamiltonian}
        \State \Return 0
    \EndIf
    \State generateAndStoreMinLeafDegreeSequences($G$) // Algorithm~\ref{alg:findMLSequences}
    \For{$v\in V(G)$}
        \State generateAndStoreMinLeafDegreeSequences($G - v$)// Algorithm~\ref{alg:findMLSequences}
    \EndFor
    \For{Sequence $L$ of $G$}
        \For{$v\in V(G)$}
            \For{Sequence $L_v$ of $G - v$}
                \State Compute transition cost
            \EndFor
            \State Store minimum transition cost over $L_v$'s
        \EndFor
        \State Store maximum of minimum transition costs over $v$'s
    \EndFor
    \State \Return $\varphi(G)$
\end{algorithmic}
\end{algorithm}

In more detail our algorithm works as follows. 
Given an input graph $G$, we first check whether it is $1$-hamiltonian using a program~\cite{GRWZ22} designed by the authors for~\cite{GRWZ24}. If the graph is $1$-hamiltonian---clearly, there are infinitely many such graphs---, then we know by the following proposition that its fault cost is~$0$.

\begin{prop}\label{fc0}
A graph has fault cost $0$ if and only if it is $1$-hamiltonian, i.e.\ $1$-leaf-guaranteed.
\end{prop}
\begin{proof} Recall that we work under the general assumption that our graphs are 2-connected, hence we do not treat $K_2$ (which is 1-leaf-guaranteed but not 1-hamiltonian). It is clear that a 1-hamiltonian graph $G$ satisfies $\varphi(G) = 0$, so let us now assume that there exists a non-1-hamiltonian graph $G$ with $\varphi(G) = 0$. The graph $G$ cannot be hamiltonian, for otherwise there would exist a vertex $x$ in $G$ such that $G - x$ is non-hamiltonian. Thus, any ml-subgraph of $G-x$ would have at least two leaves, in which case the fault cost of $G$ would be at least 2, a contradiction. Therefore $G$ is non-hamiltonian and all of its vertex-deleted subgraphs must also be non-hamiltonian, similarly to the previous situation. Let now $T$ be an optimal ml-subgraph of $G$. By the above arguments, $T$ must be a tree. Consider a leaf $b$ of $T$ and let $T_b$ be an arbitrary ml-subgraph of $G - b$. Again, by the above, $T_b$ must be a tree. Then $\varphi(G) = 0$ implies 
$$\tau(T,T_b) = |\{ v \in V(G) \setminus \{ b \} : \deg_T(v) \ne \deg_{T_b}(v) \}| = 0,$$
so
$$\sum_{v \in V(G)} \deg_T(v) = \deg_T(b) + \sum_{v \in V(G) \setminus \{ b \}} \deg_{T_b}(v),$$
which, if we set $n := |V(G)|$, is equivalent to $2(n - 1) = 1+ 2(n-2),$ a contradiction. 
\end{proof}

We now continue with the description of our algorithm. If $G$ is not $1$-hamiltonian, we will determine all distinct degree sequences of all its ml-subgraphs (see Algorithm~\ref{alg:findMLSequences}). 

\begin{algorithm}[!htb]
\caption{generateAndStoreMinLeafDegreeSequences($G$)}\label{alg:findMLSequences}
\begin{algorithmic}
    \State Create $\mathcal{L}$ // Data structure for storing the sequences
    \If{$G$ is hamiltonian} // Already known if we computed $1$-hamiltonicity
        \State Store single sequence $L$ where all degrees are $2$
        \State \Return $\mathcal{L}$
    \EndIf
    \For{each non-edge $uv$ of $G$}
        \If{$G$ has hamiltonian $uv$-path}
            \State Store sequence $L$, where $u$, $v$ have degree $1$ and $x\neq u,v$ has degree $2$
        \EndIf
    \EndFor
    \If{$\mathcal{L}$ is not empty} // $G$ is traceable
        \State \Return $\mathcal{L}$
    \EndIf
    \For{each ml-subgraph of $G$} // Determined via backtracking and using pruning rules
        \State Compute degree sequence $L$
        \If{$L\not\in\mathcal{L}$}
            \State Store $L$ in $\mathcal{L}$
        \EndIf
    \EndFor
    \State \Return $\mathcal{L}$
    
\end{algorithmic}    
\end{algorithm}

If the graph is hamiltonian (which we do not need to check again), its ml-subgraphs are the hamiltonian cycles of the graph. However, in that case there is only one distinct degree sequence, namely, the one in which every vertex has degree $2$. If the graph is not hamiltonian, using methods from~\cite{GRWZ22}, we determine for every pair of non-adjacent vertices, whether or not there is a hamiltonian path between them. As the graph is not hamiltonian, there cannot be a hamiltonian path between adjacent vertices. For every non-adjacent pair we only need to determine one hamiltonian path as for every such hamiltonian path the degrees of every vertex except for the leaves are $2$ and hence the degree sequences are the same. If we did not find any hamiltonian path in the graph, then we perform an exhaustive search for all ml-subgraphs as follows. 

Using a backtracking algorithm, we generate all spanning trees of the graph. We start with one vertex and every iteration, we choose an edge incident to the current subtree which is not forbidden and first add it to the subtree (if it does not create a cycle), while later we forbid the edge. We then do the same for this newly acquired subtree or newly acquired set of forbidden edges. For the efficiency of the algorithm, we choose an edge for which the endpoint in the subtree $v$ is most constrained.
This means that, denoting the subtree by $T$ and the subgraph of $G$ induced by the forbidden edges by $F$, we want $d_G(v) - d_T(v) - d_F(v)$ to be minimal.
Its other endpoint is the neighbour of $v$ which is most constrained in the same way, such that the edge is not forbidden and not yet in the tree. We can backtrack whenever a vertex of $G$ not yet in the tree is only incident with forbidden edges.

If the tree has size $\lvert V(G)\rvert -1$, we have a spanning tree of our graph. If it has the same number of leaves as the minimum encountered so far, we store its degree sequence in a list. If it is smaller than the minimum encountered so far, we clear the list, store this degree sequence as its first entry and keep track of this new minimum. Since we know the minimum encountered so far, we can actually prune the search as soon as the subtree has more leaves than this minimum. Due to the way we generate the spanning trees, the number of leaves can never decrease unless we backtrack.

In this way, we obtain all ml-subgraphs of $G$ and we repeat this procedure for every vertex-deleted subgraph $G-v$. However, if $v$ was a leaf of an ml-subgraph of $G$, we already know that ${\rm ml}(G)$ is an upper bound for the minimum leaf number of $G -v$ and keep track of this in order to prune earlier. 

Now that we have every distinct degree sequence, we need to compare each degree sequence of the graph with each degree sequence of the vertex-deleted subgraphs. We also use some 
impactful techniques 
to speed up these
computations. For every degree sequence of an ml-subgraph of $G$, we compare it to the distinct degree sequence of a vertex-deleted subgraph in order to determine the minimum transition cost. While checking the transition cost for two degree sequences is linear in the number of vertices, we also keep track of the leaves and degree $2$ vertices of the graph, by means of bitsets.
These low degree vertices can then be compared in constant time and we only need to iterate over the vertices of degree at least $3$, i.e.\ the branches, in the ml-subgraphs. As the number of branches is relatively small (for example of the 153\,620\,333\,545 graphs of order $12$  which are $2$-connected, only 144 of their ml-subgraphs contained four branches, while none contained more than four), the gain from this method outweighs the overhead of creating the extra bitsets and storing the degree $2$ vertices in terms of efficiency. We also apply a small optimisation
when iterating over the vertices. When the transition cost is already higher than one found earlier, we need not iterate over the remaining vertices. The minimum transition cost is then compared to those of other vertex-deleted subgraphs to determine the maximum. Finally, the minimum of these values is determined over all ml-subgraphs of $G$. 

While it is relatively easy to prove the correctness of this algorithm, we also performed various tests in order to give evidence for the correctness of the implementation. A description of these tests can be found in Appendix~\ref{app:correctness}.

\subsubsection{Computational results}

We use our implementation of Algorithm~\ref{alg:fc} to obtain counts for the number of graphs attaining each fault cost for small graphs. We do this by generating all $2$-connected graphs for a specific order using the program \texttt{geng}, which is part of the \texttt{nauty} library~\cite{MP14}. We then determine the fault cost for each of these generated graphs using our algorithm. Via \texttt{geng} it is easy and efficient to restrict the generation to graphs of at least a given girth (the \textit{girth} of a graph $G$ which is not a tree is the length of a shortest cycle in $G$) and go up to higher orders in this way. The results of our computations\footnote{
These computations were performed on a cluster of Intel Xeon Platinum 8468  (Sapphire Rapids) CPUs. The computation for order $12$ and no restriction on the girth took approximately 43 hours for the generation and 1\,312 hours for computing the fault costs. For higher girths the fault cost computation is the bottleneck.
} are summarised in Table~\ref{tab:counts_2-conn}. 
\begin{table}[!htb]
    \centering
    \setlength\tabcolsep{3.8pt} 
    \begin{tabular}{c|rrrrrrrrrrr}
         $n \backslash \varphi$ & 0 & 1 & 2 & 3 & 4 & 5 & 6 & 7 & 8 & 9 &10  \\\hline
         3 & 0&0&1&0&0&0&0&0&0&0&0 \\
         4 & 1&0&2&0&0&0&0&0&0&0&0 \\
         5 & 3&0&7&0&0&0&0&0&0&0&0 \\
         6 & 13&0&43&0&0&0&0&0&0&0&0 \\
         7 & 116&0&341&0&11&0&0&0&0&0&0 \\
         8 & 2\,009&0&5\,016&6&92&0&0&0&0&0&0 \\
         9 & 72\,529&0&119\,730&130&1\,677&0&0&0&0&0&0 \\
         10& 4\,784\,268&0&4\,926\,191&3\,930&29\,141&0&12&0&0&0&0 \\
         11& 554\,267\,470&0&345\,785\,155&133\,243&783\,029&47&137&10&0&0&0\\
         12& 111\,383\,671\,391&0&42\,204\,241\,063&6\,480\,547&25\,933\,650&2\,112&4\,98&184&0&0&0
\\\hline
         13&258\,028&0&8\,890\,460&62\,453&980\,464&425&1\,114&114&3&0&0\\
         14& 11\,388\,066&0&233\,751\,383&1\,326\,967&16\,543\,622&7\,023&17\,890&1\,236&13&2&0\\

         \hline
         15& 62&0&13\,7330&8\,265&155\,461&267&752&46&2&0&0\\
         16& 984&2&1\,508\,210&87\,456&1\,468\,221&2\,481&5\,584&323&6&0&0\\
         17& 16\,590&0&19\,026\,942&1\,089\,274&15\,726\,242&24\,957&47\,298&2\,508&33&5&2\\
         18& 327\,612&3&274\,100\,472&15\,204\,227&189\,369\,374&285\,906&459\,165&21\,688&241&61&12\\
    \end{tabular}
    \caption{Counts of how many $2$-connected graphs attain each fault cost $\varphi$ for each order $n$. The top part of the table gives counts for all $2$-connected graphs, the middle part for $2$-connected graphs of girth at least $4$, and the bottom part for $2$-connected graphs of girth at least $5$. 
    Fault costs for which the count is zero or which are not included in the table imply that no graphs of the given orders attain this fault cost, for example no graphs of order at most $12$ have a fault cost higher than $7$.
    } 
    \label{tab:counts_2-conn}
\end{table}


As can be seen from Table~\ref{tab:counts_2-conn}, graphs with fault cost~$1$ appear to be significantly rarer when compared to graphs with other low fault costs. Thus, after giving certain more general results, we will focus on structurally investigating fault cost 1 graphs in Section~\ref{subsec_fc1}. In that section we present a 14-vertex graph with fault cost~1 (originally reported in the extended abstract~\cite{GRWZ23}), see Fig.~\ref{fig:tfc1_14}. Combining this with the computational results given in Table~\ref{tab:counts_2-conn}, we obtain the following proposition, wherein $\varphi_k$ shall be the order of the smallest graph with fault cost $k$.




\begin{prop}\label{prop:min_order_for_fc}
    We have $\varphi_0 = 4, \varphi_1 \in \{ 13, 14 \}, \varphi_2 = 3, \varphi_3 = 8, \varphi_4 = 7, \varphi_5 = 11, \varphi_6 = 10, \varphi_7 = 11$, and $\varphi_8 = 13$. 
\end{prop}

An example of a most symmetric graph attaining each $\varphi_k$ can be found in the Appendix~\ref{app:hog}. All of the most symmetric graphs attaining each $\varphi$ as well as the smallest graphs\footnote{The two graphs of order $16$, girth $5$ and fault cost~$1$ can be inspected on the House of Graphs~\cite{CDG23} at \url{https://houseofgraphs.org/graphs/53049} and \url{https://houseofgraphs.org/graphs/53050}.} of girth $5$ with fault cost $1$ can also be searched on the House of Graphs~\cite{CDG23} by using the keywords ``fault cost''.


Most of the families we deal with in the sequel are of connectivity~$2$. 
It is natural to ask whether for the same orders and fault-costs also $3$-connected graphs appear. By filtering the graphs obtained from \texttt{geng} for $3$-connectivity using the program \texttt{pickg} from the \texttt{nauty}~\cite{MP14} package, we can give a similar table for the fault costs of $3$-connected graphs. See Table~\ref{tab:counts_3-conn} in Appendix~\ref{app:counts_3-conn}.

In order to obtain more examples at higher orders, we restrict our search to cubic graphs. We generated $2$-connected cubic graphs using \texttt{snarkhunter}~\cite{BGM11} and then determined the fault cost of each generated graph using our program. The results are summarised in Table~\ref{tab:counts_cubic}.


As in the general case, it seems that also in the cubic case graphs with odd fault cost are a bit rarer than the graphs with even fault cost. In particular, the smallest cubic graphs with odd fault cost have order $22$.\footnote{The most symmetric cubic graph of order $22$ with fault cost $3$ can be inspected on the House of Graphs~\cite{CDG23} at \url{https://houseofgraphs.org/graphs/53072}.} 

\begin{table}[!htb]
    \centering
    \begin{tabular}{c|rrrrr}
         $n \backslash \varphi$ & 0 & 1 & 2 & 3 & 4 \\\hline
         4& 1&0&0&0&0\\
         6& 1&0&1&0&0\\
         8& 2&0&3&0&0\\
         10& 6&0&12&0&0\\
         12& 27&0&54&0&0\\
         14& 158&0&322&0&0\\
         16& 1\,396&0&2\,478&0&0\\
         18& 16\,067&0&23\,796&0&3\\
         20& 227\,733&0&270\,061&0&24\\
         22& 3\,740\,294&0&3\,447\,110&14&209\\
         24& 68\,237\,410&0&48\,110\,143&224&1\,858\\
         26& 1\,346\,345\,025&0&726\,174\,160&3\,326&17\,841\\
         \hline
         28& 7\,352\,343\,711&0&1\,185\,463\,316&384&2\,956\\
         \hline
         30& 14\,468\,621\,439&0&153\,224\,637&0&0\\

    \end{tabular}
    \caption{Counts of fault costs for $2$-connected cubic graphs. Fault costs for which the count is zero or which are not included in the table imply that no graphs of the given orders attain this fault cost. The top part
    of the table gives counts for all $2$-connected cubic graphs, the middle part for $2$-connected cubic graphs of girth at
    least $4$, and the bottom part for $2$-connected cubic graphs of girth at least $5$.
    }
    \label{tab:counts_cubic}
\end{table}

As \texttt{snarkhunter} allows to restrict the generation to $3$-connected graphs, we also determined their fault costs. The results are summarised in Table~\ref{tab:counts_3-conn_cubic} in Appendix~\ref{app:counts_3-conn_cubic}.

For planar graphs the most efficient way of generation depends on the connectivity of the graphs. We generated $2$-connected planar graphs by generating all $2$-connected graphs via \texttt{geng} and then filtered the planar ones using \texttt{planarg} which is also part of the \texttt{nauty}~library~\cite{MP14}. We then used our program to determine the fault costs of these $2$-connected planar graphs. The results are summarised in Table~\ref{tab:counts_2-conn_planar}. For $3$-connected planar graphs, generation can be done much more efficiently by using \texttt{plantri}~\cite{BM07}. We also applied our program to these graphs and the results are summarised in Table~\ref{tab:counts_3-conn_planar} of Appendix~\ref{app:counts_3-conn_planar}. Both in the cubic and in the planar case, the restriction to $3$-connected graphs seems to restrict the fault costs, since, from the results in the appendix, it can be observed that for $3$-connected cubic or planar graphs only fault costs $0$ and $2$ appear for small orders. 

\begin{table}[!htb]
    \centering
    \begin{tabular}{c|rrrrrrrr}
        $n \backslash \varphi$ & 0 & 1 & 2 & 3 & 4 & 5 & 6 & 7\\\hline
        3& 0&0&1&0&0&0&0&0\\
        4& 1&0&2&0&0&0&0&0\\
        5& 2&0&7&0&0&0&0&0\\
        6& 7&0&37&0&0&0&0&0\\
        7& 34&0&249&0&11&0&0&0\\
        8& 246&0&2\,549&6&92&0&0&0\\
        9& 2\,526&0&32\,396&119&1\,455&0&0&0\\	
        10& 30\,842&0&489\,870&2\,682&22\,402&0&12&0\\
        11& 416\,108&0&8\,138\,729&59\,733&414\,982&38&137&10\\
        12& 5\,955\,716&0&143\,994\,243&1\,385\,013&8\,223\,342&1\,327&3\,734&184\\
    \end{tabular}
    \caption{Counts of fault costs for $2$-connected planar graphs. Fault costs for which the count is zero or which are not included in the table imply that no graphs of the given orders attain this fault cost.
    }
    \label{tab:counts_2-conn_planar}
\end{table}



\subsection{Large fault cost}

In Proposition~\ref{fc0} we characterised graphs with fault cost 0, establishing that these are exactly the 1-leaf-guaranteed graphs. Although we cannot give a characterisation, examples of graphs with fault cost 2 are easy to describe: hamiltonian, but not $1$-leaf-guaranteed graphs are clearly of this type. However, they are not the only ones, e.g.\ $K_{2,3}$ also has fault cost 2, along with various other non-hamiltonian graphs. Deciding whether graphs with some given fault cost $k$ exist is harder, even---or we might say especially---for $k=1$. Therefore, in the sequel we shall focus on a structural description of fault cost 1 graphs. But before doing so, we give a general result relating the minimum leaf number of a graph to the minimum leaf numbers of the graph's vertex-deleted subgraphs, and prove that for every non-negative integer $k$ there exists a graph with fault cost $k$.



\begin{prop}
Let $G$ be a $2$-connected graph. Then $${\rm ml}(G) - 1 \le {\rm ml}(G - v) \le {\rm ml}(G) + \Delta(G)$$ for all $v \in V(G)$. Both bounds are best possible.
\end{prop}

\noindent \emph{Proof.} If $G$ is hamiltonian we have ${\rm ml}(G) = 1$ and so ${\rm ml}(G - v) \in \{ 1, 2 \}$ for all $v \in V(G)$. So the above inequalities trivially hold. Henceforth, we assume $G$ to be non-hamiltonian, i.e.\ ${\rm ml}(G) \ge 2$. 

Put $k := {\rm ml}(G)$. If there exists a $v \in V(G)$ such that $G - v$ is hamiltonian, i.e.\ ${\rm ml}(G - v) = 1$, then $G$ is traceable and so ${\rm ml}(G) = 2$ since we are assuming that $G$ is non-hamiltonian. Yet again the inequalities hold, so we may suppose that $G - v$ is non-hamiltonian for any $v \in V(G)$. Assume there exists a vertex $v \in V(G)$ such that $G - v$ contains a spanning tree $T$ with at most $k-2$ leaves. We add $v$ and an edge incident to $v$ to $T$ and obtain a spanning tree of $G$ with at most $k - 1$ leaves, a contradiction since ${\rm ml}(G) = k$. We have proven the lower bound and it is easy to produce examples exhibiting this bound, e.g.\ a hypotraceable graph.

Now we prove the upper bound. Let $T'$ be an ml-subgraph of $G$. Note that $T'$ is a tree since we are assuming $G$ to be non-hamiltonian. Consider $v' \in V(G)$ with $\deg(v') = \Delta(G) > 1$. The disconnected graph $T' - v'$ is a forest consisting of pairwise disjoint trees $T_1, \ldots, T_{\Delta(G)}$. Since $G$ is 2-connected, between any two trees $T_i$ and $T_j$, $i \ne j$, there exists in $G - v'$ a path $P$ between a vertex in $T_i$ and a vertex in $T_j$ such that $T_i \cup P \cup T_j$ is a tree. In an analogous manner we connect all trees $T_1, \ldots, T_{\Delta(G)}$ until a spanning tree $T''$ of $G - v'$ is obtained. As the tree $T'$ has ${\rm ml}(G)$ leaves, the forest $T' - v'$ has at most ${\rm ml}(G) + \Delta(G)$ leaves. When constructing $T''$ in each step we add a path between two trees, so in each step the total number of leaves cannot increase. Thus $T''$ also has at most ${\rm ml}(G) + \Delta(G)$ leaves. See Fig.~\ref{fault-cost_k} for a graph $G$---indeed, an infinite family of graphs---showing that, in general, this bound cannot be improved: it is not difficult to check that ${\rm ml}(G) = \Delta(G)$ and ${\rm ml}(G - v) = 2\Delta(G)$ for $v$ as defined in Fig.~\ref{fault-cost_k}. \hfill $\Box$

\begin{figure}[!htb]
\begin{center}
\includegraphics[height=44mm]{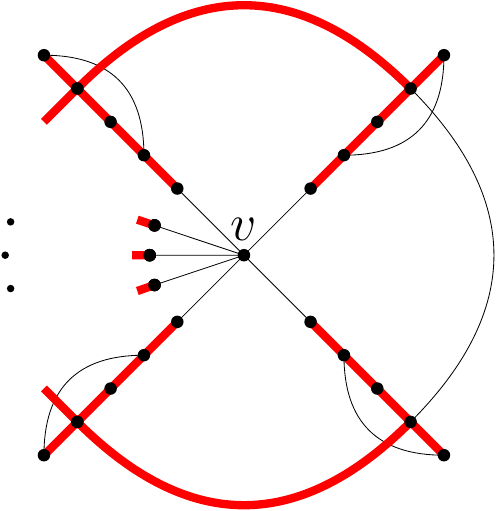} 
\caption{
A 2-connected graph $G$ with maximum degree and minimum leaf number $k := \deg(v) \ge 3$. In $G - v$, any spanning tree has at least $2k$ leaves. An ml-subgraph of $G - v$ with exactly $2k$ leaves is emphasised. 
}\label{fault-cost_k}
\end{center}
\end{figure}

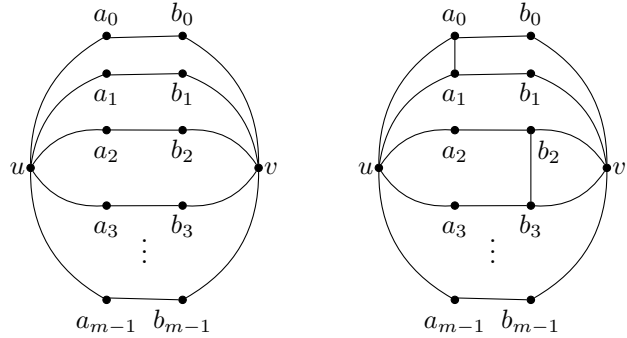
\begin{figure}
    \centering
    \begin{tikzpicture}[ fo/.style={draw, circle, fill=black, minimum size={0.15cm}, inner sep=0cm, scale=0.65, label={[label distance=-3pt]#1\strut}},scale=0.5] \node[fo=left:$u$] (u) at (-1,0) {};
        \node[fo=right:$v$] (v) at (5,0) {};

        \node[fo=above:$a_0$] (a0) at (1,3.5) {};
        \node[fo=below:$a_1$] (a1) at (1,2.5) {};
        \node[fo=below:$a_2$] (a2) at (1,1) {};
        \node[fo=below:$a_3$] (a3) at (1,-1) {};
        \node[fo=below:$a_{m-1}$] (am) at (1,-3.5) {};
        
        \node[fo=above:$b_0$] (b0) at (3,3.5) {};
        \node[fo=below:$b_1$] (b1) at (3,2.5) {};
        \node[fo=below:$b_2$] (b2) at (3,1) {};
        \node[fo=below:$b_3$] (b3) at (3,-1) {};
        \node[fo=below:$b_{m-1}$] (bm) at (3,-3.5) {};

        \node at (2, -2) {\vdots};

        \draw (u) to[bend left] (a0) -- (b0) to[bend left] (v);
        \draw (u) to[bend left] (a1) -- (b1) to[bend left] (v);
        \draw (u) to[bend left] (a2) -- (b2) to[bend left] (v);
        \draw (u) to[bend right] (a3) -- (b3) to[bend right] (v);
        \draw (u) to[bend right] (am) -- (bm) to[bend right] (v);
    \end{tikzpicture}
    \qquad
        \begin{tikzpicture}[ fo/.style={draw, circle, fill=black, minimum size={0.15cm}, inner sep=0cm, scale=0.65, label={[label distance=-3pt]#1\strut}},scale=0.5]
        \node[fo=left:$u$] (u) at (-1,0) {};
        \node[fo=right:$v$] (v) at (5,0) {};

        \node[fo=above:$a_0$] (a0) at (1,3.5) {};
        \node[fo=below:$a_1$] (a1) at (1,2.5) {};
        \node[fo=below:$a_2$] (a2) at (1,1) {};
        \node[fo=below:$a_3$] (a3) at (1,-1) {};
        \node[fo=below:$a_{m-1}$] (am) at (1,-3.5) {};
        
        \node[fo=above:$b_0$] (b0) at (3,3.5) {};
        \node[fo=below:$b_1$] (b1) at (3,2.5) {};
        \node[fo=below right:$b_2$] (b2) at (3,1) {};
        \node[fo=below:$b_3$] (b3) at (3,-1) {};
        \node[fo=below:$b_{m-1}$] (bm) at (3,-3.5) {};

        \node at (2, -2) {\vdots};

        \draw (u) to[bend left] (a0) -- (b0) to[bend left] (v);
        \draw (u) to[bend left] (a1) -- (b1) to[bend left] (v);
        \draw (u) to[bend left] (a2) -- (b2) to[bend left] (v);
        \draw (u) to[bend right] (a3) -- (b3) to[bend right] (v);
        \draw (u) to[bend right] (am) -- (bm) to[bend right] (v);
        \draw (a0) -- (a1);
        \draw (b2) -- (b3);
    \end{tikzpicture}
    \caption{Left-hand side (a): The graph $G_m$, $m\geq 3$, with fault cost $2\lfloor m/2 \rfloor + 2$. Right-hand side (b): The graph $H_m$, $m\geq 5$ with fault cost $2\lfloor m/2 \rfloor - 1$.}
    \label{fig:Gm_and_Hm}
\end{figure}

\begin{thm}
For any non-negative integer $k$ there exists a graph with fault cost exactly $k$.
\end{thm}

\begin{proof}
    We first describe a family of graphs with even fault cost $k\geq 4$. We handle the small and odd fault costs later. The family can be obtained by taking the multigraph on two vertices and $m$ edges between these two vertices, and subdividing each edge twice. See Fig.~\ref{fig:Gm_and_Hm}(a). More formally, let $m\geq 3$ and let $G_m$ be the graph with vertex set $V:= \{u,v,a_0,\ldots,a_{m-1}, b_0,\ldots, b_{m-1}\}$ and edge set $E:= \{ua_i, vb_i, a_ib_i\}_{i=0}^{m-1}$. We now show that $G_m$ has fault cost $m+2$ for even $m$ and fault cost $m+1$ for odd $m$. In particular, the odd case gives us a graph with fault cost $4$ when $m = 3$. 

    Let $T$ be an ml-subgraph of $G_m$. We recall that for a vertex $v$ in $T$, its \textit{$T$-degree} is the degree $v$ has in $T$. Since $T$ is connected, there is a pair $a_i,b_i$ with $i\in\{0,\ldots, m-1\}$ which both have $T$-degree $2$. At most one such pair can exist as trees are acyclic; and since $T$ is connected the remaining pairs $a_i, b_i$ have one vertex of $T$-degree $1$ and one vertex of $T$-degree $2$. The number of leaves is minimised when $u$ and $v$ are not leaves in $T$.
    Hence, $T$ has $m-1$ leaves. We define $\alpha:= \lvert L(T)\cap \{a_0,\ldots, a_{m-1}\}\rvert$ and $\beta := \lvert L(T)\cap \{b_0,\ldots, b_{m-1}\}\rvert$. Without loss of generality, we assume that $a_0$ and $b_0$ are both of $T$-degree $2$.
    We will now look at the ml-subgraphs of the vertex-deleted subgraph $G_m-x$ of $G_m$, where $x \in V(G_m)$.

    Suppose that $x = u$. Then $G_m-x$ is a tree with $m$ leaves and hence this tree $T_x$ is the only ml-subgraph $G_m-x$. We have that $\tau(T,T_x) = 2+2\beta$ as $a_0$
    and $v$ change degrees (since $\alpha \geq 1$) and for every pair $a_i,b_i$ where $b_i$ is a leaf in $T$, the degrees change.
    Similarly, when $x = v$, $\tau(T, T_x) = 2+2\alpha$. Note that $2+2\alpha\geq 4$ and $2+2\beta \geq 4$ for any $m\geq 3$. 
    
    For any other vertex $x$ of $G_m$, $\tau(T, T_x) \leq 3$ for an ml-subgraph $T_x$ of $G - x$ which minimises this value.
    Indeed, let $x = a_0$, then an ml-subgraph has $m-1$ leaves. One can be obtained from $T$ by changing the degrees of $b_0$, $u$ and $a_i$, where $a_i$ was a leaf of $T$. Similarly, when $x = b_0$, one can be obtained from $T$ using only three changes. When $x$ is a leaf of $T$, an ml-subgraph exists with only one change and when $x$ is a vertex of degree $2$, not $a_0$ or $b_0$, an ml-subgraph exists with three changes ($u, v$ and the remaining neighbour of $x$ in $T$).

    Therefore, $\varphi_T(G_m) = \max \{ 2+2\alpha, 2+2\beta \}$, which is minimised over $T$ when $\alpha = \lfloor (m-1)/2\rfloor$ and $\beta=\lceil (m-1)/2\rceil$ or vice versa. Hence, $\varphi(G_m) = m+1$ when $m$ is odd and $\varphi(G_m) = m+2$ when $m$ is even. 

    We now handle odd fault costs using the previous family, but with two added edges. See Fig.~\ref{fig:Gm_and_Hm}(b). More formally, let $m\geq 5$ and $H_m:= G_m + a_0a_1+b_2b_3$. Then $H_m$ has fault cost $m - 2$ when $m$ is odd and $m-1$ when $m$ is even, as we now prove.
    
    Any ml-subgraph $T$ of $H_m$ has exactly $m-3$ leaves and contains the path $vb_0a_0a_1b_1$ or $vb_1a_1a_0b_0$, where $a_0$ and $a_1$ have $T$-degree $2$, as well as the path $ua_2b_2b_3a_3$ or $ua_3b_3b_2a_2$, where $b_2$ and $b_3$ have $T$-degree $2$ and $T$ has a pair of vertices $a_i, b_i$, with $i\in\{4,\ldots, m-1\}$, which are both of $T$-degree $2$. Indeed, in a spanning tree $T'$ of $H_m$ the set of vertices $\{ a_i, b_i \}_{i=4}^{m-1}$ contains
    at least $m-5$ leaves of $T'$, and at least $m-4$ leaves when there is no pair $a_i, b_i$, with $i\in\{4,\ldots, m-1\}$, such that both $a_i$ and $b_i$ have $T'$-degree $2$. The set $\{ a_i, b_i \}_{i=0}^3$ contains at least two leaves of $T'$, but would introduce a cycle if $\{ a_i, b_i \}_{i=4}^{m-1}$ contains a pair $a_i, b_i$ such that both $a_i$ and $b_i$ have $T'$-degree $2$, and one of $a_0, a_1, b_2, b_3$ would be of $T'$-degree $3$. Therefore, an ml-subgraph $T$ has the structure described above.
    
    Due to symmetry, we can assume that $T$ contains the paths $vb_0a_0a_1b_1$ and $ua_2b_2b_3a_3$ and that $a_4$ and $b_4$ are both of $T$-degree $2$. We define $\alpha$ and $\beta$ as before.

    We now look at the ml-subgraphs $T_x$ of the vertex-deleted subgraphs $H_m-x$ of $H_m$. Suppose $x = u$; again any ml-subgraph $T_x$ will contain either $vb_0a_0a_1b_1$ or $vb_1a_1a_0b_0$. By our assumption on $T$, the option which minimises $\tau(T, T_x)$ always takes the former path. There are three options left for $T_x$ as $b_2$ and $b_3$ can either both have $T_x$-degree $2$;
    or $b_2$ can have $T_x$-degree $3$ and $b_3$ can have $T_x$-degree $2$;
    or $b_2$ can have $T_x$-degree $2$ and $b_3$ can have $T_x$-degree $3$.
    However, to minimise the transition cost, both $b_2$ and $b_3$ need to be of $T_x$-degree $2$. This gives $\tau(T, T_x) = 3+2(\beta-1)$, since the degree changes in $v, a_2, a_4$ and every pair $a_i, b_i$ with $i\in \{5,\ldots, m-1\}$ in which $b_i$ was a leaf in $T$.
    Due to symmetry, if $x = v$ then the ml-subgraph $T_x$ which minimises the transition cost gives $\tau(T, T_x) = 3+2(\alpha - 1)$.
    Note that when $m\geq 6$, the maximum of these values is at least $4$ and that when $m = 5$, both of them are equal to $3$.
    
    For any other vertex $x$ of $H_m$ we have $\tau(T, T_x)\leq 4$ for some ml-subgraph $T_x$ of $H_m-x$. When $m = 5$, we even have $\tau(T, T_x)\leq 3$. Indeed, let $x = a_0$. Then an ml-subgraph of $H_m - a_0$ has $m-2$ leaves. One can be obtained from $T$ by changing the degrees of $u$ and $b_0$. The same holds for $b_2$. If $x = a_1$, an ml-subgraph has $m-2$ leaves and one can be obtained from $T$ by changing the degrees of $a_0$ and $v$. The same holds for $b_3$. If $x = b_0$, an ml-subgraph has $m-3$ leaves and one can be obtained from $T$ by changing the degrees of $a_0$ and $b_1$. The same holds for $a_2$. If $x$ is a leaf of $T$, then an ml-subgraph of $H_m-x$ can be obtained from $T$ using one degree change in the neighbour of that leaf. If $x = a_4$, then an ml-subgraph of $H_m-x$ has $m-3$ leaves if $m\geq 6$ and $m-2$ leaves if $m = 5$. The former can be solved by attaching a leaf which is not $b_1$ or $a_3$ to $u$ or $v$. This can be done using at most four changes. The latter case can be dealt with in two changes adding the edge $a_0u$. The same holds when $x = b_4$. Finally, if $x$ is any of the remaining vertices, i.e.\ $a_i$ or $b_i$ with $i\in \{5,\ldots, m-1\}$ with $T$-degree $2$, then one can find an ml-subgraph of $H_m - x$ by only changing the degree of $u$ and $v$.

    Therefore, $\varphi_T(H_m) = \max\{3+2(\alpha-1), 3+2(\beta-1)\}$, which is minimised over $T$ when $\alpha = \lfloor (m-3)/2\rfloor$ and $\beta = \lceil (m-3)/2\rceil$ or vice versa. Hence, $\varphi(H_m) = m-2$ when $m$ is odd and $\varphi(H_m) = m-1$ when $m$ is even.

    This proves the statement for graphs with fault cost $k\geq 3$. An example of a graph with fault cost $2$ can be found in Appendix~\ref{app:fc2_example}. A search for fault cost $1$ graphs is handled in the sequel. Graphs with vanishing fault cost exist due to Proposition~\ref{fc0}.
\end{proof}

\subsection{Small fault cost}

As above experiments show, among graphs with small fault cost, graphs with fault cost 0 or 2 are ubiquitous and graphs with fault cost 1 or 3 are rarer. In the next two sections we therefore focus on these two small odd fault cost cases.

\subsubsection{Fault cost 3}

In this section we describe a construction for obtaining graphs of fault cost $3$ and show there exist infinitely many cubic graphs with fault cost $3$.

Let $H$ be a not necessarily 2-connected graph. We say $H$ is a \emph{Type~1} graph if it contains pairwise distinct vertices $v,w,x,y$ such that all of the following hold. 

\begin{enumerate}[label=\normalfont{(\roman*)}]
\item There is no hamiltonian $vw$-path in $H$;

\item There is a $vx$-path $P$ and a $wy$-path $Q$ with $V(P), V(Q)$ partitioning $V(H)$;

\item $H$ has a hamiltonian $v$-path and a hamiltonian $w$-path whose other endpoint lies in $\{x, y\}$;

\item $H - v$ and $H - w$ each contain a hamiltonian path with one endpoint in $\{ v, w\}$ and the other endpoint in $\{x, y\}$;

\item $H - u$ has a hamiltonian $vw$-path for any vertex $u \in V(H) \setminus \{v,w\}$;

\item $H$ has a $3$-leaf spanning tree with $v,w$ leaves and the remaining leaf and branch in $\{x,y\}$.
\end{enumerate}

Let $H'$ be a graph containing distinct vertices $v,w$ such that there is a hamiltonian $vw$-path, and for every $u \in V(H') \setminus \{ v, w \}$ the graph $H' - u$ admits a hamiltonian $vw$-path or a $3$-leaf spanning tree with $v$ and $w$ as leaves. We say $H'$ is a \emph{Type~2} graph.


\begin{thm}\label{thm:construction_fc3}
    For any integer $k\geq 3$, any Type~1 graph of order $n_0$ and any $k$ choices of Type~2 graphs of order $n_1,\ldots, n_k$, respectively, there is a graph $G$ of order $\sum_{i=0}^kn_i$ 
    with $\varphi(G) = 3$. Moreover, if the Type~2 graphs have the property that for any vertex $y$ and $v,w$ as defined above, for any hamiltonian $vy$-path there is no hamiltonian $wy$-path, or vice versa, then this also holds for $k = 2$.
\end{thm}
\begin{proof}
    Let $G$ be the graph obtained by the disjoint union of a Type~1 graph $H_0$,
    where we denote $v, w$, as defined above, by $v_0, w_0$, and $k\ge 2$
    Type~2 graphs $H_1, \ldots, H_k$, where we denote $v,w$ in $H_i$, as defined above, by $v_i, w_i$, by adding the edges $w_iv_{i+1}$ for $i\in \{0,\ldots, k-1\}$ and $w_kv_0$. For the sake of convenience, we define $H_{k+1} := \emptyset$. We prove that $G$ has fault cost $3$. 

    The ml-subgraphs of $G$ are its hamiltonian paths. Indeed, $G$ is non-hamiltonian since $H_0$ has no hamiltonian $v_0w_0$-path by (i), but $G$ is traceable, since every $H_i$, $i>0$, has a hamiltonian $v_iw_i$-path and $H_0$ has disjoint $v_0$- and $w_0$-paths spanning all vertices of $H_0$ by (ii).

    Let us fix some ml-subgraph $\mathfrak{p}$ of $G$ and denote its leaves by $x$ and $y$. Due to (i) at least one of its leaves is a vertex in $H_0$. We assume for now that both $x$ and $y$ lie in $H_0$ and look at the ml-subgraphs of vertex-deleted subgraphs $G - u$ of $G$.

    Suppose $u \in \{ v_0, w_0 \}$.
    Then $G - u$ is not hamiltonian as we remove a vertex from a $2$-separator of $G$.
    However, $G - u$ is traceable since $H_0 - u$ has a hamiltonian $v_0$- or $w_0$-path by (iv). Therefore its ml-subgraphs are hamiltonian paths. Depending on the choice of path and the choice of $x$ and $y$,
    this means that either
    zero, 
    one or three vertices of $H_0$ change degree in an ml-subgraph $S_u$ of $G - u$. Since $S_u$ has a leaf in $H_1$ or $H_k$, we get $\tau(\mathfrak{p},S_u) \in \{1,2,4\}$. Note that $\tau(T, S_u) = 1$ can only happen when $u$ is $x$ or $y$ since we assume that $x,y\in V(H_0)$.

    Consider $u\in V(H_0)\setminus \{v_0, w_0\}$. Since $H_0 - u$ has a hamiltonian $v_0w_0$-path by (v), $G - u$ is hamiltonian so we have $\tau(\mathfrak{p}, S_u) \in \{ 1,2 \}$ in any ml-subgraph $S_u$ of $G - u$.  

    Suppose $u\in \{v_1, w_k\}$.
    Then $G - u$ is non-hamiltonian as we have removed a vertex from a $2$-separator of $G$.
    Since $H_i$, $i\in \{1,k\}$, has a hamiltonian $v_iw_i$-path because it is a Type~2 graph, we also obtain a hamiltonian path in $H_i$ by removing $u$.
    As $H_0$ has a hamiltonian $v_0$- or $w_0$-path by (iii), we get that $G - u$ is traceable and that its ml-subgraphs are the hamiltonian paths.
    Depending on the choice of path and the choice of $x$ and $y$, this means that either one or three vertices of $H_0$ change degree in an ml-subgraph $S_u$ of $G - u$. As $S_u$ has a leaf in $H_1$ or $H_k$, we have that
    $\tau(\mathfrak{p}, S_u) \in \{2,4\}$. 

    Assume $u\in \{v_i, w_i\}\setminus \{v_1, w_k\}$ for an arbitrary but fixed $i\in \{1,\ldots, k\}$. Then a vertex of $H_i$ and of $H_{i-1}$ or $H_{i+1}$ 
    will have degree $1$ in an ml-subgraph of $G - u$. Since $H_0$ has no hamiltonian $v_0w_0$-path by (i), $G - u$ is not traceable. However, since $H_0$ has a $3$-leaf spanning tree in which $v_0$ and $w_0$ are leaves by (vi), the ml-subgraphs of $G - u$ are trees with exactly three leaves. Similarly, we can find $3$-leaf spanning
    trees of $G - u$ of which the restriction to $H_0$ is a hamiltonian $v_0$- or $w_0$-path of $H_0$ by (iii).
    An ml-subgraph $S_u$ of $G - u$ either has one, two, three or four vertices in $H_0$ which change degree with respect to $\mathfrak{p}$, depending on the choice of $3$-leaf spanning tree of $H_0$ (or hamiltonian $v_0$- or $w_0$-path of $H_0$) and the choice of $x$ and $y$. Therefore, $\tau(\mathfrak{p}, S_u)\in \{3,4,5,6\}$. We note that any $3$-leaf spanning tree of $H_0$ yielding an ml-subgraph of $G-u$ must have $v_0$ and $w_0$ as leaves.

    Finally, suppose $u\in V(H_i)\setminus\{v_i, w_i\}$ for an arbitrary but fixed $i\in \{1,\ldots, k\}$. Then $G - u$ is still not hamiltonian. If $H_i - u$ has a hamiltonian $v_iw_i$-path, then there is an ml-subgraph $S_u$ of $G - u$ for which $\tau(\mathfrak{p}, S_u) = 0$. If there is no such path, then $H_i$ has a $3$-leaf spanning tree with leaves $v_i$ and $w_i$, as it is a Type~2 graph, and there exist ml-subgraphs $S_u$ in $G - u$ for which $\tau(\mathfrak{p}, S_u) = 2$.

    We have now shown that for any hamiltonian path $\mathfrak{p}$ of $G$ with both leaves in $H_0$ that $\varphi_\mathfrak{p}(G)\geq 3$.

    It also holds that for any hamiltonian $x'y'$-path $\mathfrak{p}'$ with $x'$ in $H_0$ and $y'$ in $H_1$ or $H_k$ we have that $\varphi_{\mathfrak{p}'}(G)\geq 3$. Note that at least one leaf must be in $H_0$ by (i).

    Indeed, taking $u\in \{v_i, w_i\}\setminus\{v_1, w_k\}$ for an arbitrary but fixed $i\in \{1,\ldots, k\}$, an ml-subgraph $S_u$ of $G - u$ has leaves in $H_i$ and $H_{i-1}$ or $H_{i+1}$ and a leaf and a branch in $H_0$. If $k\geq 3$, we can take $u$ such that $y$ cannot be one of these leaves and we get $\tau(\mathfrak{p}',S_u)\geq 4$. If $k = 2$, we have by the extra assumption on $H_1$ and $H_2$ that we can take $u$ such that $y$ is not a leaf in any ml-subgraph $S_u$ of $G-u$. Hence, $\tau(\mathfrak{p}', S_u)\geq 3$. 

    If we now take $\mathfrak{p}$ to be a hamiltonian path in $G$ with leaves $x,y$ for which (i)--(vi) hold. 
    Then by the above we have $\varphi_\mathfrak{p}(G) = 3$. 

    As, by definition, $\varphi(G) = \min_{S \in {\cal S}_{{\rm ml}}(G)} \varphi_S(G)$, this shows that $\varphi(G) = 3$.
    
\end{proof}

\begin{cor}\label{cor:family_cubic_fc3}
    There exist infinitely many cubic graphs with fault cost $3$.
\end{cor}
\begin{proof}
    Consider an integer $k\geq 2$ and let $H_0$ be the graph of Fig.~\ref{fig:family_cubic_fc3}. It is tedious but straightforward to verify that it is a Type~1 graph for the indicated $v$, $w$, $x$, and $y$. Let $H_i$, for $i\in \{1,\ldots, k\}$, be copies of the graph in Fig.~\ref{fig:family_cubic_fc3}(b). Again it is easy to verify that each $H_i$ is a Type~2 graph. The result follows by Theorem~\ref{thm:construction_fc3}. It also holds for $k = 2$ as the extra condition of the theorem is satisfied for the graph in Fig.~\ref{fig:family_cubic_fc3}(b).
\end{proof}

\begin{figure}[!htb]
    \centering
    \tikzstyle{fo} = [draw, circle, fill=black, minimum size={0.15cm}, inner sep=0cm, scale=0.65]
    \begin{tikzpicture}[scale=0.2]
        \node[fo, label=below:$v$] (w) at (0,0) {};
        \node[fo] (1) at (4,0) {};
        \node[fo] (2) at (8,0) {};
        \node[fo] (3) at (12,0) {};
        \node[fo] (4) at (16,0) {};
        \node[fo, label=below:$w$] (v) at (20,0) {};

        \draw (w) -- (1) -- (2) -- (3) -- (4) -- (v);

        \node[fo, label=above:$x$] (5) at (5,5) {};
        \node[fo] (6) at (10,8) {};
        \node[fo, label=above:$y$] (7) at (15,5) {};
        
        \draw (w) -- (5) -- (6) -- (7) -- (v);

        \node[fo] (8) at (10, 4) {};
        
        \draw (6) -- (8);
        \draw (1) -- (8);
        \draw (4) -- (8);

        \draw (3) -- (5);
        \draw (2) -- (7);

    \end{tikzpicture}
    \qquad
    \begin{tikzpicture}[scale=0.8]
        \node[fo, label=below:$v$] (v) at (0,0) {};
        \node[fo] (1) at (1,0) {};
        \node[fo] (2) at (2,0) {};
        \node[fo, label=below:$w$] (w) at (3,0) {};

        \draw (v) -- (1) -- (2) -- (w);
        
        \node[fo] (3) at (1,1.5) {};
        \node[fo] (4) at (2,1.5) {};

        \draw (v) -- (3) -- (4) -- (w);
        \draw (1) -- (4);
        \draw (2) -- (3);
    \end{tikzpicture}
    \caption{Subgraphs used in the proof of Corollary~\ref{cor:family_cubic_fc3}. Left-hand side (a): A Type~1 graph. Right-hand side (b): A Type~2 graph also satisfying the extra condition of Theorem~\ref{thm:construction_fc3}.}
    \label{fig:family_cubic_fc3}
\end{figure}
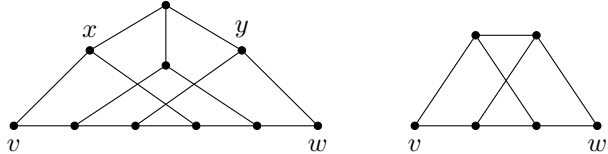

\subsubsection{Fault cost 1} \label{subsec_fc1}

When looking for graphs with fault cost 1, the first candidates to go over is the family of all graphs with minimum leaf number 2, that is: traceable, but not hamiltonian graphs. For traceable fault cost 1 graphs we will use the shorthand term \emph{tfc1 graphs} in the sequel. Tfc1 graphs $G$ must be 2-leaf-guaranteed: they cannot be hamiltonian, for otherwise the fault cost would be 0 or 2, showing that $\ml (G) =2$. Now we have to prove that $\ml (G-x)\leq 2$ for all $x\in V(G)$. Assume this is not true for some $x$: then any spanning tree $S_x$ of $G-x$ has at least 3 leaves and therefore there would also exist a vertex of $S_x$ of degree at least 3, thus we would have at least 4 vertices of $S_x$ not of degree 2, contradicting the tfc1 property of $G$. In fact, it is not difficult to see that in every tfc1 graph there exist at most two vertices whose deletion yields a hamiltonian graph:


Let $G$ be tfc1. We already know that $G$ is 2-leaf-guaranteed, so an optimal ml-subgraph $P$ of $G$ must be a hamiltonian path. Let the end-vertices of $P$ be $a_1$ and $a_2$. If we delete a vertex $v$ in $G$ different from $a_1$ and $a_2$, we cannot obtain a hamiltonian graph, as then every ml-subgraph of $G - v$ would be a hamiltonian cycle and the degrees of $a_1$ and $a_2$ would be different in $P$ and any ml-subgraph of $G - v$. This contradicts the fact that $G$ has fault cost 1 or that $P$ was an optimal ml-subgraph.



\begin{cl} \label{elsoclaim} Let $G$ be a $2$-leaf-stable graph. Then $G$ has fault cost $1$ if and only if there exist vertices $a_1,a_2\in V(G)$, such that there exists a hamiltonian $a_1a_2$-path in $G$ and for any vertex $x\in V(G)-a_1-a_2$ there exists a hamiltonian $a_1a_2$-path in $G-x$ as well.
\end{cl}


\begin{proof} Let $P$ be a hamiltonian $a_1a_2$-path in $G$. We point out that $P$ is an ml-subgraph of $G$ because $G$ is non-hamiltonian. In $G - a_i$ the hamiltonian path $P - a_i$ is an ml-subgraph and $\tau(P,P - a_i) = 1$. For all $v \in V(G) \setminus \{ a_1, a_2 \}$, let $M$ be any ml-subgraph of $G - v$. Then $M$ is not a hamiltonian cycle as $G$ is 2-leaf-stable. If $M$ is chosen to be a hamiltonian $a_1a_2$-path, which we know exists, then $\tau(P,M) = 0$. So $\varphi_P(G) = 1$. Since $G$ is not 1-hamiltonian, by Claim~1 we have $\varphi(G) = 1$.

In order to prove the other direction, let $P$ be an optimal ml-subgraph of $G$. Since $G$ is 2-leaf-stable, $P$ is a hamiltonian path; let $a_1,a_2$ be the endvertices of $P$ and let $x$ be an arbitrary vertex in $V(G)-a_1-a_2$. 
$G - x$ is non-hamiltonian, so there exists a minimum leaf spanning tree $P'$ of $G-x$ (where $P'$ is a hamiltonian path, since $G$ is 2-leaf-stable), such that at most one of the vertices of $G-x$ has different degrees in $P$ and $P'$. We show that the endvertices of $P'$ are $a_1$ and $a_2$. Assume to the contrary that the endvertices are $b,c$, such that $\{a_1,a_2\} \ne \{b,c\}$, and  w.l.o.g. also assume $a_1\ne b, a_1\ne c$. Then $a_1$ and (at least) one of the vertices $b,c$ have different degrees in $P$ and $P'$, a contradiction. 
\end{proof}

Graphs that are 2-leaf-stable and have fault cost 1 are called \textit{$2$-lsfc1} in the sequel. 

\begin{remark} The vertices $a_1$ and $a_2$ do not have to be unique.
\end{remark}


\noindent An immediate corollary of the previous claim is that tfc1 graphs can have at most two vertices of degree 2 (by deleting a neighbour of a vertex $x\not\in \{a_1,a_2\}$ of degree 2, we cannot have a hamiltonian $a_1a_2$-path). On the other hand, $a_1$ and $a_2$ might have degree 2 (and also any degree greater than 1), as we shall see later. This also means that tfc1 graphs need not be 3-connected. Now we are dealing with tfc1 graphs of connectivity 2, for which we need the following notions. Let $X$ be a 2-separator of a graph $G$ and let $H$ be one of the components of $G-X$. Then $G[V(H) \cup X]$ is called a \emph{$2$-fragment} of $G$, and $X$ is called the \emph{attachment} of $H$. Let $G_1$ and $G_2$ be graphs, such that there exist two vertices $x,y$, such that  $\{x,y\} = V(G_1)\cap V(G_2)$. Then $G_1:G_2$ denotes the graph obtained by \emph{gluing together} $G_1$ and $G_2$ at the vertices $x,y$, i.e.\ the graph with vertex set $V(G_1)\cup V(G_2)$ and edge set $E(G_1)\cup E(G_2)$. The next claim follows easily from Claim~\ref{elsoclaim}.   

\begin{cl} \label{masodikclaim} Let $G$ be a $2$-lsfc1 graph, $a_1,a_2$ as described in Claim~\ref{elsoclaim}, and let $\{x,y\}$ be a separator of $G$. Then the following hold.  \begin{enumerate}[label=\normalfont{(\roman*)}]
\item $\{a_1,a_2\} \cap \{x,y\} = \emptyset$.
\item There are exactly two different $2$-fragments $G_1, G_2$ of $G$ with attachment $\{x,y\}$, such that $a_i \in V(G_i)$ for $i=1,2$.
\item There exists a hamiltonian $a_ix$-path in $G_i-y$ and a hamiltonian $a_iy$-path in $G_i-x$ for $i=1,2$.
\item For any $v\in V(G_i) \setminus \{ a_i \}$ there exists a hamiltonian $a_ix$- or $a_iy$-path in at least one of the graphs $G_i-v$, $G_i-x-v$, $G_i-y-v$ for both $i=1$ and $i = 2$.
\item If $xy \not\in E(G)$ then $G+xy$ is also a tfc1 graph. 
\end{enumerate}
\end{cl}

Note that statement (iii) of this claim is actually a special case of statement (iv), but it is worth mentioning it in its own right because of its corollaries.

We now focus on proving the existence of $2$-lsfc1 graphs. In order to do so, let us observe that we may suppose that $x$ and $y$ are neighbours in a 2-fragment of a $2$-lsfc1 graph, by point 5 of Claim~\ref{masodikclaim}. Assuming this, statement (iv) of Claim~\ref{masodikclaim} becomes much easier to handle:

\begin{cl} \label{harmadikclaim} Let $G, G_1, G_2, x,y,a_1,a_2$ be as described in Claim~\ref{masodikclaim}, such that $xy\in E(G)$. Then there exists a hamiltonian $a_ix$- or $a_iy$-path in $G_i$ and also in $G_i-v$ for any $v\in V(G_i) \setminus \{ a_i \}$ for $i=1,2$. 
\end{cl}

\noindent It seems natural that a graph fulfilling the property described in Claim~\ref{harmadikclaim} is not necessarily a 2-fragment of some tfc1 graph. So we need further  properties. Somewhat surprisingly, these properties are easy to describe and might not even be needed (at least not in both fragments). 

Let $H$ be a connected graph, $a,x,y\in V(H)$, and $xy\in E(H)$. Consider the following properties of the quadruple $(H,a,x,y)$.  

\medskip \noindent ($P_0$) For any $v\in V(H)-a$ there exists a hamiltonian $ax$- or $ay$-path in $H$ and also in $H-v$.

\medskip \noindent ($Q_1$) There exists no hamiltonian $xy$-path in $H$. 

\medskip \noindent ($Q_2$) For any $v\in V(H) - a$ there exists no hamiltonian $xy$-path in $H-v$.

\medskip \noindent The quadruple $(H,a,x,y)$ is said to be a \emph{weak fragment} if it fulfills ($P_0$), a \emph{medium fragment} if it fulfills ($P_0$) and ($Q_1$), and finally a \emph{strong fragment} if it fulfills ($P_0$), ($Q_1$), and ($Q_2$).

For convenience's sake, a graph $H$ can also be called a weak/medium/strong fragment if there exist vertices $a,x,y\in V(H)$, such that $(H,a,x,y)$ is a weak/medium/strong fragment. Note that medium fragments are also weak fragments and strong fragments are both medium and weak fragments as well. By gluing together such fragments we can obtain tfc1 graphs:

\begin{thm} \label{G5} Let $(G_1,a_1,x,y)$ and $(G_2,a_2,x,y)$ be weak fragments. If both of them are also medium or one of them is also strong, then $G:=G_1:G_2$ is a tfc1 graph.
\end{thm}

\begin{proof}
$G$ is easily seen to be traceable: there exists a hamiltonian $a_1x$- or $a_1y$-path $P^1$ in $G_1$, let us assume w.l.o.g.\ the former. There also exists a hamiltonian $a_2x$-path $P^2$ in $G_2-y$. Now $P:=P^1\cup P^2$ is a hamiltonian $a_1a_2$-path of $G$. 

The non-hamiltonicity of $G$ follows immediately from the fact that (at least) one of $G_1$ and $G_2$ is medium and therefore there is no hamiltonian $xy$-path  in (at least) one of $G_1$ and $G_2$. 

Thus $P$ is an ml-subgraph of $G$ and it is enough to show that $\varphi_P(G)\leq 1$ ($G$ is not hamiltonian, so $\varphi (G) \geq 1$ by Proposition~\ref{fc0}). In order to do this we have to show that for each vertex $v\in V(G)$ there exists an ml-subgraph $S_v$ of $G-v$ such that $\tau(P,S_v) \leq 1$. 

Let us consider first the case $v=a_i$ for $i=1,2$. 
If $G - a_i$ is hamiltonian, then any ml-subgraph $S_{a_i}$ of $G-a_i$ is a hamiltonian cycle, thus $\tau(P,S_{a_i}) = 1$.  If $G - a_i$ is not hamiltonian, let $S_{a_i} := P - a_1$. Now $\tau(P,S_{a_i}) = 1$ is obvious again.

Now let $v=x$ (the case $v=y$ is the same). Since both $G_1$
and $G_2$ are weak fragments, there exist hamiltonian $a_iy$-paths $P_i$ of $G_i-x$ for $i=1,2$. Now for the hamiltonian $a_1a_2$-path $P':= P_1 \cup P_2$ of $G-x$ we have $\tau(P,P') = 0$. Note that $P'$ is an ml-subgraph of $G-x$, as $G-x$ is not hamiltonian (actually, not even 2-connected). 

Finally, let $v \in V(G)\setminus \{a_1, a_2,x,y\}$. We show that $G-v$ is not hamiltonian. W.l.o.g.\ let us assume that $v\in G_1$. In order for $G-v$ to be hamiltonian we need a hamiltonian $xy$-path in both $G_1-v$ and $G_2$. The existence of the former implies that $G_1$ is not a strong fragment, while the existence of the latter implies that $G_2$ is not a medium fragment, contradicting the choice of $G_1$ and $G_2$. Now we know that any hamiltonian path of $G-v$ is an ml-subgraph. Now by the weak fragment property of $G_1$ we have a hamiltonian $a_1x$- or $a_1y$-path $P^1$ of $G_1-v$, w.l.o.g.\ let us assume that the endvertices of $P^1$ are $a_1$ and $x$. Since $G_2$ is also a weak fragment, there exists a hamiltonian $a_2x$-path $P^2$ of $G_2-y$. Now for the hamiltonian $a_1a_2$ path $P':= P^1 \cup P^2$ of $G-v$ we have $\tau(P,P') = 0$, finishing the proof.
\end{proof}

Recall that in every tfc1 graph there exist at most two vertices whose deletion yields a hamiltonian graph. Note that the proof of Theorem~\ref{G5} shows that in this construction only $a_1$ and $a_2$ might possess this property.
  
Weak fragments are easy to find, e.g.\ any complete graph of order at least 3 is a weak fragment; and by adding an edge to a non-complete weak fragment we also obtain a weak fragment. Examples are given in Fig.~\ref{weak}.  
\begin{figure}
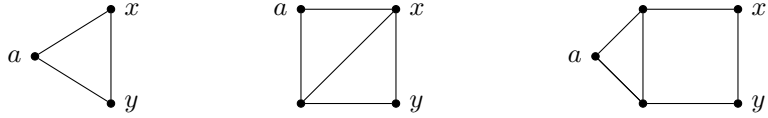

\begin{center}
\tikz[ fo/.style={draw, circle, fill=black, minimum size={0.15cm}, inner sep=0cm, scale=0.65},scale=0.5]{

\node [fo, label=left:\hskip0mm$a$ ] (1) at (0,0) {};
\node [fo, label=right:\hskip0mm$x$ ] (2) at (2,1.25) {};
\node [fo, label=right:\hskip0mm$y$ ] (3) at (2,-1.25) {};

\graph { (3) -- (1) -- (2) -- (3)};

\node [fo ] (1) at (7,-1.25) {};
\node [fo, label=right:\hskip0mm$y$ ] (2) at (9.5,-1.25) {};
\node [fo, label=left:\hskip0mm$a$ ] (3) at (7,1.25) {};
\node [fo, label=right:\hskip0mm$x$ ] (4) at (9.5,1.25) {};

\graph { (1) -- (2), (3) -- (4) -- (1) -- (3), (2) -- (4)};

\node [fo ] (1) at (16,-1.25) {};
\node [fo, label=right:\hskip0mm$y$ ] (2) at (18.5,-1.25) {};
\node [fo] (3) at (16,1.25) {};
\node [fo, label=right:\hskip0mm$x$ ] (4) at (18.5,1.25) {};
\node [fo, label=left:\hskip0mm$a$ ] (5) at (14.75,0) {};

\graph { (1) -- (2), (3) -- (4), (1) -- (3), (5) -- (1) -- (5) -- (3), (2) -- (4)};

}

\end{center}
\vspace{-5mm}
\caption{Examples of weak fragments. \label{weak}}
\end{figure}
\noindent However, weak fragments are not enough to build tfc1 graphs using Theorem~\ref{G5} as we need at least one medium fragment. This is somewhat harder to describe. Examples based on Petersen's graph are shown in Fig.~\ref{medium} (we leave the verification that these are indeed medium fragments to the reader).
\begin{figure}
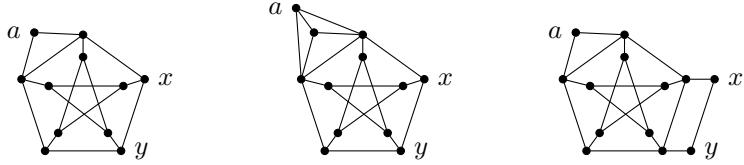

\begin{center}
\tikz[ fo/.style={draw, circle, fill=black, minimum size={0.15cm}, inner sep=0cm, scale=0.65},scale=0.5]{

\node [fo] (1) at (0,0) {};
\node [fo, label=right:\hskip0mm$y$ ] (2) at (2,0) {};
\node [fo] (3) at (-0.62,1.9) {};
\node [fo, label=right:\hskip0mm$x$ ] (4) at (2.62,1.9) {};
\node [fo] (5) at (1,3.08) {};

\node [fo] (6) at (0.345,0.475) {};
\node [fo] (7) at (2.62/4+2/2,1.9/4) {};
\node [fo] (8) at (1/4-0.62/4,3.08/4+1.9/2) {};
\node [fo] (9) at (1/4+2/4+2.62/2,3.08/4+1.9/2) {};
\node [fo] (10) at (1,1.9+1.18/2) {};

\node [fo, label=left:\hskip0mm$a$ ] (11) at (-0.62/2+1/2 - 1.18/2.5 , 1.9/2+3.08/2 + 1.62/2.5) {};

\graph { (1) -- (2) -- (4) -- (5) -- (3) -- (1) -- (6) -- (9) -- (8) -- (7) -- (2), (6) -- (10) -- (7), (3) -- (8), (5) -- (10), (4) -- (9), (3) -- (11) -- (5)};

}
\hskip9mm \tikz[ fo/.style={draw, circle, fill=black, minimum size={0.15cm}, inner sep=0cm, scale=0.65},scale=0.5]{

\node [fo] (1) at (0,0) {};
\node [fo, label=right:\hskip0mm$y$ ] (2) at (2,0) {};
\node [fo] (3) at (-0.62,1.9) {};
\node [fo, label=right:\hskip0mm$x$ ] (4) at (2.62,1.9) {};
\node [fo] (5) at (1,3.08) {};

\node [fo] (6) at (0.345,0.475) {};
\node [fo] (7) at (2.62/4+2/2,1.9/4) {};
\node [fo] (8) at (1/4-0.62/4,3.08/4+1.9/2) {};
\node [fo] (9) at (1/4+2/4+2.62/2,3.08/4+1.9/2) {};
\node [fo] (10) at (1,1.9+1.18/2) {};

\node [fo ] (11) at (-0.62/2+1/2 - 1.18/2.5 , 1.9/2+3.08/2 + 1.62/2.5) {};
\node [fo, label=left:\hskip0mm$a$ ] (12) at (-0.62/2+1/2 - 1.18/1.25 , 1.9/2+3.08/2 + 1.62/1.25) {};

\graph { (1) -- (2) -- (4) -- (5) -- (3) -- (1) -- (6) -- (9) -- (8) -- (7) -- (2), (6) -- (10) -- (7), (3) -- (8), (5) -- (10), (4) -- (9), (3) -- (11) -- (5), (3) -- (12) -- (5), (11) -- (12)};

}
\hskip9mm \tikz[ fo/.style={draw, circle, fill=black, minimum size={0.15cm}, inner sep=0cm, scale=0.65},scale=0.5]{

\node [fo] (1) at (0,0) {};
\node [fo, label=right:\hskip0mm$$ ] (2) at (2,0) {};
\node [fo] (3) at (-0.62,1.9) {};
\node [fo, label=right:\hskip0mm$$ ] (4) at (2.62,1.9) {};
\node [fo] (5) at (1,3.08) {};

\node [fo] (6) at (0.345,0.475) {};
\node [fo] (7) at (2.62/4+2/2,1.9/4) {};
\node [fo] (8) at (1/4-0.62/4,3.08/4+1.9/2) {};
\node [fo] (9) at (1/4+2/4+2.62/2,3.08/4+1.9/2) {};
\node [fo] (10) at (1,1.9+1.18/2) {};

\node [fo, label=left:\hskip0mm$a$ ] (11) at (-0.62/2+1/2 - 1.18/2.5 , 1.9/2+3.08/2 + 1.62/2.5) {};

\node [fo, label=right:\hskip0mm$y$ ] (12) at (2.75,0) {};
\node [fo, label=right:\hskip0mm$x$ ] (13) at (3.37,1.9) {};

\graph { (1) -- (2) -- (4) -- (5) -- (3) -- (1) -- (6) -- (9) -- (8) -- (7) -- (2), (6) -- (10) -- (7), (3) -- (8), (5) -- (10), (4) -- (9), (3) -- (11) -- (5), (2) -- (12) -- (13) -- (4)};

}
\end{center}
\vspace{-5mm}
\caption{Examples of medium fragments.\label{medium}}
\end{figure}
Using two (not necessarily different) graphs of Fig.~\ref{medium} and Theorem~\ref{G5}, we obtain tfc1 graphs, see Fig.~\ref{2ls1fc1}. 
\begin{figure}
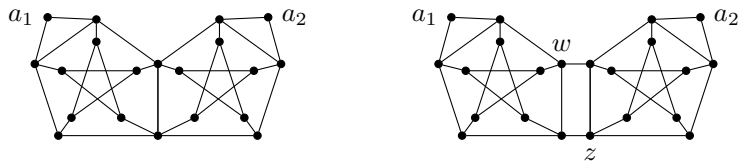

\begin{center}
\tikz[ fo/.style={draw, circle, fill=black, minimum size={0.15cm}, inner sep=0cm, scale=0.65},scale=0.5]{

\node [fo] (1) at (0,0) {};
\node [fo, label=below:\color{white}\hskip0mm$z$ ] (2) at (2.62,0) {};
\node [fo] (3) at (-0.62,1.9) {};
\node [fo, label=right:\hskip0mm ] (4) at (2.62,1.9) {};
\node [fo] (5) at (1,3.08) {};

\node [fo] (6) at (0.345,0.475) {};
\node [fo] (7) at (2.62/4+2/2,1.9/4) {};
\node [fo] (8) at (1/4-0.62/4,3.08/4+1.9/2) {};
\node [fo] (9) at (1/4+2/4+2.62/2,3.08/4+1.9/2) {};
\node [fo] (10) at (1,1.9+1.18/2) {};

\node [fo, label=left:\hskip0mm$a_1$ ] (11) at (-0.62/2+1/2 - 1.18/2.5 , 1.9/2+3.08/2 + 1.62/2.5) {};

\node [fo] (21) at (2*2.62,0) {};
\node [fo] (23) at (5.86,1.9) {};
\node [fo] (25) at (4.24,3.08) {};

\node [fo] (26) at (2.62-0.345+2.62,0.475) {};
\node [fo] (27) at (2*2.62-2.62/4-2/2,1.9/4) {};
\node [fo] (28) at (2*2.62-1/4+0.62/4,3.08/4+1.9/2) {};
\node [fo] (29) at (2*2.62-1/4-2/4-2.62/2,3.08/4+1.9/2) {};
\node [fo] (30) at (2*2.62-1,1.9+1.18/2) {};

\node [fo, label=right:\hskip0mm$a_2$ ] (31) at (2*2.62+0.62/2-1/2 + 1.18/2.5 , 1.9/2+3.08/2 + 1.62/2.5) {};

\graph { (1) -- (2) -- (4) -- (5) -- (3) -- (1) -- (6) -- (9) -- (8) -- (7) -- (2), (6) -- (10) -- (7), (3) -- (8), (5) -- (10), (4) -- (9), (3) -- (11) -- (5)};

\graph { (21) -- (2) -- (4) -- (25) -- (23) -- (21) -- (26) -- (29) -- (28) -- (27) -- (2), (26) -- (30) -- (27), (23) -- (28), (25) -- (30), (4) -- (29), (23) -- (31) -- (25)};

}
\hskip9mm \pgfmathsetmacro{\elt}{0.75} \tikz[ fo/.style={draw, circle, fill=black, minimum size={0.15cm}, inner sep=0cm, scale=0.65},scale=0.5]{

\node [fo] (1) at (0,0) {};
\node [fo, label=right:\hskip0mm ] (2) at (2.62,0) {};
\node [fo] (3) at (-0.62,1.9) {};
\node [fo, label=above:\hskip0mm$w$ ] (4) at (2.62,1.9) {};
\node [fo] (5) at (1,3.08) {};

\node [fo] (6) at (0.345,0.475) {};
\node [fo] (7) at (2.62/4+2/2,1.9/4) {};
\node [fo] (8) at (1/4-0.62/4,3.08/4+1.9/2) {};
\node [fo] (9) at (1/4+2/4+2.62/2,3.08/4+1.9/2) {};
\node [fo] (10) at (1,1.9+1.18/2) {};

\node [fo, label=left:\hskip0mm$a_1$ ] (11) at (-0.62/2+1/2 - 1.18/2.5 , 1.9/2+3.08/2 + 1.62/2.5) {};

\node [fo] (21) at (2*2.62+\elt,0) {};
\node [fo, label=below:\hskip0mm$z$ ] (22) at (2.62+\elt,0) {};
\node [fo] (23) at (5.86+\elt,1.9) {};
\node [fo, label=right:\hskip0mm ] (24) at (2.62+\elt,1.9) {};
\node [fo] (25) at (4.24+\elt,3.08) {};

\node [fo] (26) at (2.62-0.345+2.62+\elt,0.475) {};
\node [fo] (27) at (2*2.62-2.62/4-2/2+\elt,1.9/4) {};
\node [fo] (28) at (2*2.62-1/4+0.62/4+\elt,3.08/4+1.9/2) {};
\node [fo] (29) at (2*2.62-1/4-2/4-2.62/2+\elt,3.08/4+1.9/2) {};
\node [fo] (30) at (2*2.62-1+\elt,1.9+1.18/2) {};

\node [fo, label=right:\hskip0mm$a_2$ ] (31) at (2*2.62+0.62/2-1/2 + 1.18/2.5+\elt , 1.9/2+3.08/2 + 1.62/2.5) {};

\graph { (1) -- (2) -- (4) -- (5) -- (3) -- (1) -- (6) -- (9) -- (8) -- (7) -- (2), (6) -- (10) -- (7), (3) -- (8), (5) -- (10), (4) -- (9), (3) -- (11) -- (5)};

\graph { (21) -- (22) -- (24) -- (25) -- (23) -- (21) -- (26) -- (29) -- (28) -- (27) -- (22), (26) -- (30) -- (27), (23) -- (28), (25) -- (30), (24) -- (29), (23) -- (31) -- (25), (2) -- (22) -- (24) -- (4) };

}
\end{center}
\vspace{-5mm}
\caption{Examples of tfc1 graphs. \label{2ls1fc1}}
\end{figure}

\noindent Next we would like to find strong fragments, which is obviously even harder than finding medium ones. First we characterise 2-lsfc1 graphs with a separator $\{x,y\}$, such that $x$ and $y$ are neighbours. The next theorem shows that these can only be obtained in the way described in Theorem~\ref{G5}. 

\begin{thm} \label{G6} Let $G,x,y,a_1,a_2,G_1,G_2$ be as described in Claim~\ref{elsoclaim} and Claim~\ref{masodikclaim} and let $xy \in E(G)$. Then  
$(G_1,a_1,x,y)$ and $(G_2,a_2,x,y)$ are weak fragments and either both of them are also medium or one of them is also strong.
\end{thm}

\begin{proof}
By Claim~\ref{harmadikclaim},  both $(G_1,a_1,x,y)$ and $(G_2,a_2,x,y)$ are weak fragments. Let us assume now that one of the fragments, say $(G_2,a_2,x,y)$ is not medium, that is there exists a hamiltonian $xy$-path $P^2$ of $G_2$. In order to finish the proof we just have to show that $(G_1,a_1,x,y)$ is a strong fragment, that is there is no hamiltonian path between $x$ and $y$  nor in $G_1$, neither in $G_1-v$ for any $v \in V(G_1)-a_1$. Actually these are easy to see. The union of $P^2$ and a hamiltonian $xy$-path of $G_1$ would be a hamiltonian cycle of $G$, while the union of $P^2$ and a hamiltonian $xy$-path of $G_1-v$ would be a hamiltonian cycle of $G-v$, both contradicting the 2-leaf-stability of $G$.
\end{proof}

\noindent Let us consider now (say) the first tfc1 graph of Fig.~\ref{2ls1fc1} and its 2-separator $X$ consisting of the neighbours of $a_1$. By Theorem~\ref{G6} (which can be used, since the vertices in $X$ are neighbours and $G - a_1$ and $G- a_2$ are easily seen to be non-hamiltonian), the 2-fragments with attachment $X$ are weak fragments, and it is obvious that $K_3$ (one of the fragments) is not a medium fragment, therefore the other one (which is just the first graph of the figure with $a_1$ deleted) must be a strong fragment.  



\begin{cor}
    There exist infinitely many graphs with fault cost $1$.
\end{cor}
\begin{proof}
We have seen that all complete graphs are medium fragments. Gluing these together with a strong fragment (whose existence we have just seen) we obtain infinitely many tfc1 graphs.   
\end{proof}

\noindent If they exist, graphs of order $13$ or $14$ with fault cost $1$ must, by Table~\ref{tab:counts_2-conn}, contain a triangle. Actually, we did find such graphs (of order 14, but none of order 13), as reported in the extended abstract~\cite{GRWZ23} containing some of our initial results. In Fig.~\ref{fig:tfc1_14} we reproduce one of these graphs. Using a computer it is easy to check the required properties, but---as it is stated, but not proved in~\cite{GRWZ23}---there also exists a technique generalising  Theorems~\ref{G5} and~\ref{G6}, from which these immediately follow. These more general theorems are not included here, but might be subject of a follow up paper focusing mainly on the fault cost 1 case.    

\begin{minipage}{0.9\linewidth}
    \vspace{\intextsep}
    \captionsetup{type=figure}  
    \begin{center}
        \includegraphics[height=25mm]{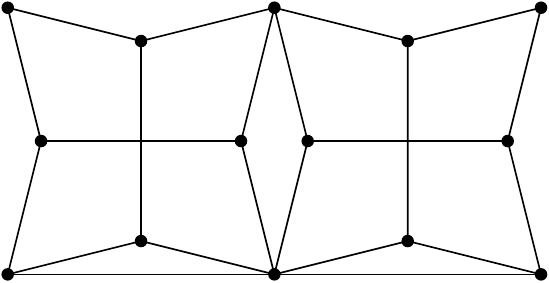}\\
        \caption{A 14-vertex graph with fault cost 1.
        }\label{fig:tfc1_14} 
    \end{center}
    \vspace{\intextsep}
\end{minipage}

In this section we have assumed that the vertices of attachment are neighbours in a tfc1 2-fragment in order to make the construction of tfc1 graphs easier. However, there might be tfc1 2-fragments without this property, thus the following questions arise naturally. Are there  tfc1 graphs with a 2-separator consisting of non-neighbouring vertices? If so, do we have infinitely many? Are there tfc1 graphs, such that \emph{all} 2-separators consist of non-neighbouring vertices? Again: if so, are there infinitely many?

The first question can be immediately answered in the affirmative: the separator $\{ w,z \}$ of the second graph of Fig.~\ref{2ls1fc1} possesses this property. This graph is obtained by gluing together the first and third medium fragments of Fig.~\ref{medium}. Notice that we can create infinitely many medium fragments by (say) substituting $a_1$ and its two neighbours (that form a $K_3$) of the (say) first graph of Fig.~\ref{2ls1fc1} with a different complete graph. This process works because of Theorems~\ref{G5} and~\ref{G6}. By gluing together these medium fragments with the third medium fragment of Fig.~\ref{medium} we obtain infinitely many tfc1 graphs with a 2-separator of non-neighbouring vertices, answering the second question positively as well. Actually, we can even create tfc1 graphs with any given number of such 2-separators using the following straightforward lemma.

\begin{lem}
Let $(H,a,x,y)$  be a weak fragment and let $H'$ be the graph obtained from $H$ by adding the vertices $x'$ and $y'$ and the edges $xx',x'y',y'y$. Then $(H',a,x',y')$  is also a weak fragment, moreover if $H$ is medium/strong then $H'$ is also medium/strong. 
\end{lem}

\medskip

The third question can be answered affirmatively as well, since the graph of Fig.~\ref {fig:tfc1_14} possesses the desired property. In order to obtain infinitely many such graphs however, we need more, like the aforementioned generalisations of Theorems~\ref{G5} and~\ref{G6}.  
\section{Open problems}\label{sec:probs}

We end this paper with two natural problems on the structurally particularly challenging graphs with fault cost 1.

\medskip

\noindent \textbf{Problem 1.} Up to now we have been discussing  constructions based on 2-fragments, thus all of our tfc1 graphs are of connectivity 2. It is natural to ask whether 3-connected tfc1 graphs exist. If they do, is there a characterization for (say) the connectivity 3 case? To move even a bit further we might also ask whether $k$-leaf-guaranteed graphs with fault cost 1 exist for $k\geq 3$. 

\medskip

\noindent \textbf{Problem 2.} Are there cubic graphs with fault cost 1?

\bigskip

\section*{Acknowledgements}
Several of the computations for this work were carried out using the supercomputer infrastructure provided by the VSC (Flemish Supercomputer Center), funded by the Research Foundation Flanders (FWO) and the Flemish Government.

\bigskip

\appendix

\section{Appendix}

\subsection{Figure in the proof of Theorem~\ref{thm:ind-subgr}}\label{app:figure_proof_thm2}

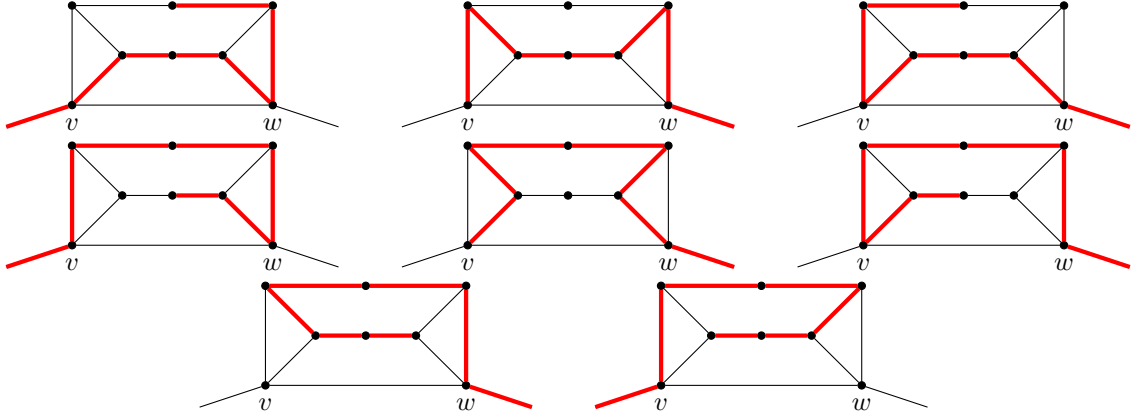
\begin{figure}[H]
    \centering
    \newcommand{\s}{0.33} 
    \tikzstyle{fo} = [draw, circle, fill=black, minimum size={0.15cm}, inner sep=0cm, scale=0.65]
    \newcommand{\graphXiEight}[1]{
        \begin{tikzpicture}[scale=\s]
            \node[fo, label=below:$v$] (v) at (0,0) {};
            \node[fo] (1) at (0,4) {};
            \node[fo] (2) at (2,2) {};
            \node[fo, label=below:$w$] (w) at (8,0) {};
            \node[fo] (3) at (8,4) {};
            \node[fo] (4) at (6,2) {};
            \node[fo] (5) at (4,4) {};
            \node[fo] (6) at (4,2) {};
            \node[] (v') at (-3,-1) {};
            \node[] (w') at (11,-1) {};
            \draw (v) -- (w);
            \draw (v) -- (v');
            \draw (w) -- (w');
            \draw (v) -- (1) -- (2) -- (v);
            \draw (w) -- (3) -- (4) -- (w);
            \draw (1) -- (5) -- (3);
            \draw (2) -- (6) -- (4);
            #1
        \end{tikzpicture}
    }
    
    \tikzstyle{rededge} = [color=red, ultra thick]
    
    \graphXiEight{\draw[rededge]  (5) -- (3) -- (w) -- (4) -- (6) -- (2) -- (v) -- (v');}
    \graphXiEight{\draw[rededge]  (v) -- (1) -- (2) -- (6) -- (4) -- (3) -- (w) -- (w');}
    \graphXiEight{\draw[rededge]  (5) -- (1) -- (v) -- (2) -- (6) -- (4) -- (w) -- (w');}
    \graphXiEight{\draw[rededge]  (6) -- (4) -- (w) -- (3) -- (5) -- (1) -- (v) -- (v');}
    \graphXiEight{\draw[rededge]  (v) -- (2) -- (1) -- (5) -- (3) -- (4) -- (w) -- (w');}
    \graphXiEight{\draw[rededge]  (6) -- (2) -- (v) -- (1) -- (5) -- (3) -- (w) -- (w');}
    \graphXiEight{\draw[rededge]  (4) -- (6) -- (2) -- (1) -- (5) -- (3) -- (w) -- (w');}
    \graphXiEight{\draw[rededge]  (v') -- (v) -- (1) -- (5) -- (3) -- (4) -- (6) -- (2);}
    \caption{Hamiltonian paths in vertex-deleted subgraphs of $\Xi_8$.}
    \label{fig:Xi_8_K1-traceable}
\end{figure}

\subsection{Example illustrating the fault cost definition}\label{app:fc2_example}

Consider the graph $\Xi_9$ shown in Fig.~\ref{fig:xi}. It is not difficult to check that $\Xi_9$ is the smallest 2-leaf-guaranteed graph, both in terms of order and size. As a 2-leaf-guaranteed graph, $\Xi_9$ is non-hamiltonian but traceable, i.e.\ its minimum leaf number is 2, and all of its vertex-deleted subgraphs are traceable.

\begin{figure}[H]
\begin{center}
\includegraphics[height=20mm]{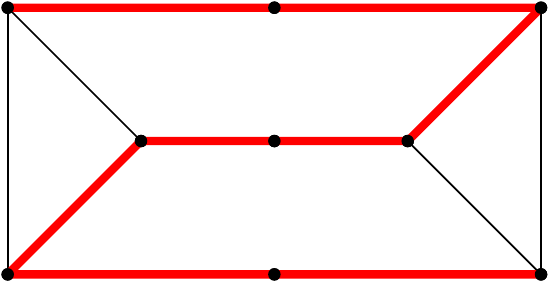}\\
\caption{The graph $\Xi_9$. Its minimum leaf number is 2.
In bold red an ml-subgraph $S$ in $\Xi_9$ is shown.}\label{fig:xi}
\end{center}
\end{figure}

Denote the trees given in the first row of Fig.~\ref{fig:xi_ml_subgraphs} with $S_1, S_2, S_3, S_4$. We have $\tau(S,S_1) = 1$, $\tau(S,S_2) = \tau(S,S_3) = \tau(S,S_4) = 3$. Therefore, if $v_1$ fails, the transition cost from $S$ to an ml-subgraph in $\Xi_9 - v_1$ is at least 1. If $v_2$ fails, the transition cost from $S$ to an ml-subgraph in $\Xi_9 - v_2$ is 4, and for every $i \in \{ 3, 4, 5 \}$, if $v_i$ fails, the transition cost from $S$ to an ml-subgraph in $\Xi_9 - v_i$ is 2. By symmetry, this covers all cases. Thus, for the ml-subgraph $S$ specified in Fig.~\ref{fig:xi}, we have $$\max_{v \in V(\Xi_9)} \min_{S_v \in {\cal S}_{{\rm ml}}(\Xi_9 - v)} \tau(S,S_v) = 4.$$

\begin{figure}[H]
\begin{center}
\includegraphics[height=72mm]{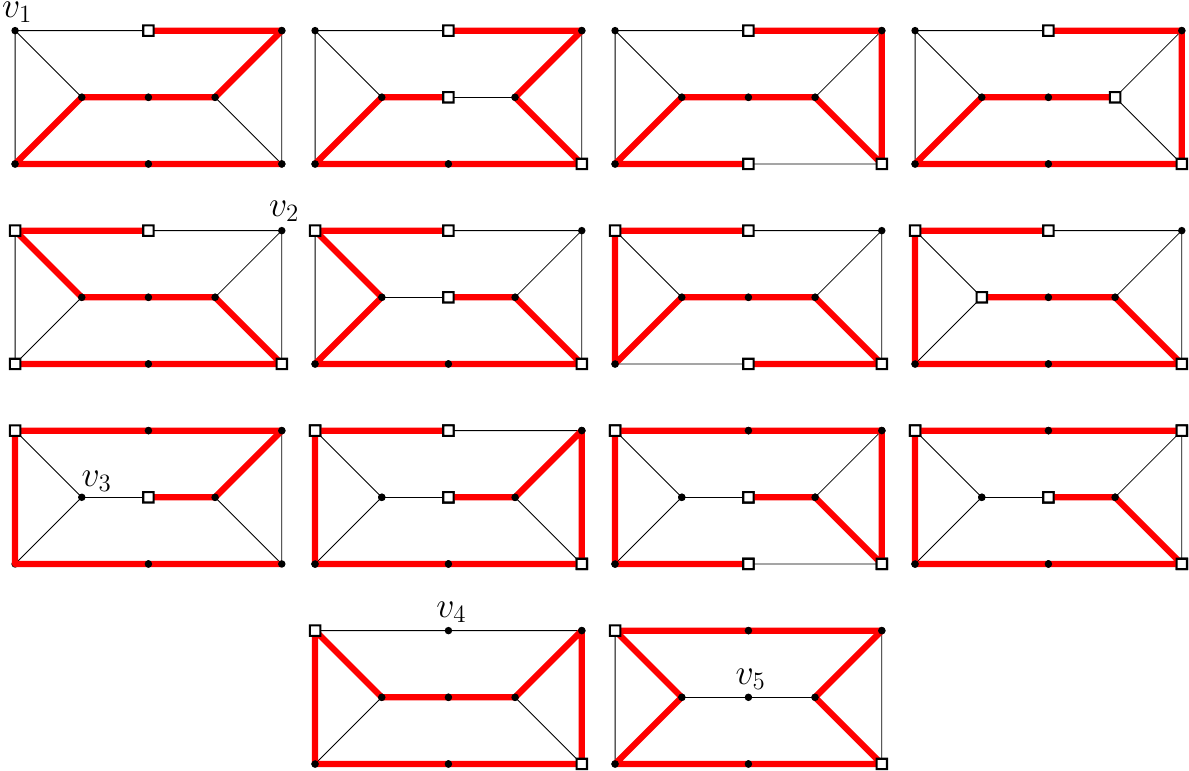}\\
\caption{Ignoring symmetric cases, the above shows all ml-subgraphs $S_v$ in $\Xi_9 - v$: twelve cases if $v \in \{ v_1, v_2, v_3 \}$ and two cases if $v \in \{ v_4, v_5 \}$. A vertex $w \ne v$ is depicted as a white square whenever $\deg_S(w) \ne \deg_{S_v}(w)$, so the number of such vertices is $\tau(S, S_v)$.}\label{fig:xi_ml_subgraphs}
\end{center}
\end{figure}

 Performing this straightforward analysis for all other ml-subgraphs of $\Xi_9$, we obtain that $\varphi(\Xi_9) = 2$. An optimal subgraph in $\Xi_9$, i.e.\ an ml-subgraph realising this minimum fault cost, is shown in Fig.~\ref{fig:xi_ml_subgraph_vertex_deleted}, left-hand side. Finally, we point out that $\Xi_9$ does contain an ml-subgraph $S$ and a vertex $v$ such that $\Xi_9 - v$ contains an ml-subgraph $S_v$ with $\tau(S,S_v) = 0$, see Fig.~\ref{fig:xi_ml_subgraph_vertex_deleted}. Graphs with vanishing fault cost are characterised in Proposition~\ref{fc0}.

\begin{figure}[H]
\begin{center}
\includegraphics[height=22mm]{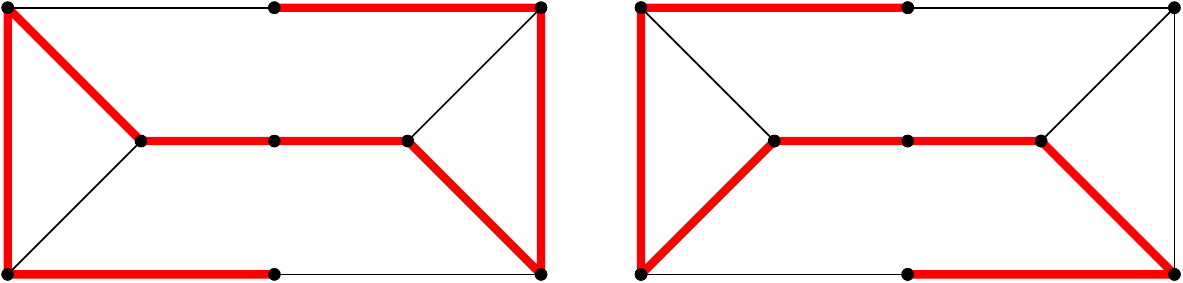}\\
\caption{An ml-subgraph of $\Xi_9$ (left-hand side) and an ml-subgraph of a vertex-deleted subgraph of $\Xi_9$ (right-hand side). The two subgraphs shown in this figure have transition cost 0.
}\label{fig:xi_ml_subgraph_vertex_deleted}
\end{center}
\end{figure}

\subsection{Correctness tests}\label{app:correctness}
While it is relatively easy to prove the correctness of Algorithm~\ref{alg:fc}, we also performed various tests to verify the correctness of the implementation. Our implementation of the algorithm is open source software and can be found on GitHub~\cite{GRWZ25},
where it can be verified and used by other researchers.

First of all, we verified that the examples of Corollary~\ref{cor:family_cubic_fc3}, indeed have fault cost $3$ up to $k = 5$. Our program determined that this was indeed the case.

Next, we verified for the family of Theorem~\ref{fault-cost_k} that the fault costs obtained by our program are the same as what is stated in the proof of the theorem. We checked this for both constructions up to order~$12$.

We also verified with our program that the two examples of Fig.~\ref{2ls1fc1} indeed have fault cost $1$.

Algorithm~\ref{alg:fc} can easily be adapted to determine the minimum leaf number of the given graphs. We compared the output of our program for the minimum leaf number to a second independent implementation of an algorithm for determining the minimum leaf number by Floor Van de Steene~\cite{Va20}. 
Both programs were in agreement for all of our checks. In particular, we checked $2$-connected graphs up to order $9$, $2$-connected graphs of girth at least $4$ up to order $11$, $2$-connected graphs of girth at least $5$ up to order $15$, $3$-connected graphs of girth at least $4$ up to order $12$, $3$-connected graphs of girth at least $5$ up to order $16$, $2$-connected cubic graphs up to order $18$, $2$-connected cubic graphs of girth at least $4$ up to order $20$, $2$-connected cubic graphs of girth at least $5$ up to order $22$ and $3$-connected planar graphs up to order $10$.

By Proposition~\ref{fc0} all $1$-hamiltonian graphs have fault cost $0$. Previously, we had already implemented an algorithm which can determine if a graph is $1$-hamiltonian (the source code of this program can be found on GitHub~\cite{GRWZ22}). We used this program to filter the $1$-hamiltonian ones in various classes of graphs and then we checked for these graphs whether our program correctly detects that their fault cost is $1$. Everything we checked was in agreement. In particular, we checked $2$-connected graphs up to order 11, $2$-connected graphs of girth at least $4$ up to order 13, $2$-connected graphs of girth at least $5$ up to order 16, 
$2$-connected cubic graphs up to order $22$,
cubic graphs of girth at least $5$ up to order $24$,
and $3$-connected planar graphs up to order $12$.

\subsection{Graphs realising the minimum of Proposition~\ref{prop:min_order_for_fc}}\label{app:hog}

In Proposition~\ref{prop:min_order_for_fc} we provided the order of the smallest graphs attaining fault cost $k$ for $k\leq 8$, with the exception of $k = 1$. For every $k$, the most symmetric graphs can be found on the House of Graphs~\cite{CDG23} by searching for the keywords ``fault cost''. We give an example of each in Fig.~\ref{fig:min_order_for_fc}.

\noindent
\begin{minipage}{\linewidth} \centering
    \vspace{\intextsep}
    \captionsetup{type=figure}
    \newcommand{\scf}{3}
    \centering
    \tikzstyle{fo}=[draw, circle, fill=black, minimum size={0.15cm}, inner sep=0cm, scale=0.65]
    \begin{minipage}{0.24\linewidth}\centering
        \begin{tikzpicture}[scale=1*\scf]
            \node[fo] (1) at (0,0) {};
            \node[fo] (2) at (0,1) {};
            \node[fo] (3) at (1,1) {};
            \node[fo] (4) at (1,0) {};
            \draw (1) -- (2) -- (3) -- (4) -- (1);
            \draw (1) -- (3);
            \draw (2) -- (4);
        \end{tikzpicture}
        \captionof*{figure}{$\varphi_0 = 4$}
    \end{minipage}
    \begin{minipage}{0.24\linewidth}\centering
        \begin{tikzpicture}[scale=1*\scf]
            \node[fo] (1) at (0,0) {};
            \node[fo] (2) at (1,0) {};
            \node[fo] (3) at (0.5,1) {};
            \draw (1) -- (2) -- (3) -- (1);
        \end{tikzpicture}
        \captionof*{figure}{$\varphi_2 = 3$}
    \end{minipage}
    \begin{minipage}{0.24\linewidth}\centering
        \begin{tikzpicture}[scale=.666*\scf]
            \node[fo] (1) at (0,0) {};
            \node[fo] (2) at (1,0) {};
            \node[fo] (3) at (0.25,.75) {};
            \node[fo] (4) at (.75,.75) {};
            \node[fo] (5) at (0.5,1.5) {};
            \node[fo] (6) at (-0.25, .75) {};
            \node[fo] (7) at (1.25, .75) {};
            \node[fo] (8) at (0.5, 0.375) {};
            \draw (1) -- (2) -- (8) -- (4) -- (5) -- (3) -- (1) -- (8);
            \draw (1) -- (6) -- (5);
            \draw (2) -- (7) -- (5);
        \end{tikzpicture}
        \captionof*{figure}{$\varphi_3 = 8$}
    \end{minipage}
    \begin{minipage}{0.24\linewidth}\centering
        \begin{tikzpicture}[scale=.666*\scf]
            \node[fo] (1) at (0,0) {};
            \node[fo] (2) at (1,0) {};
            \node[fo] (3) at (0.25,.75) {};
            \node[fo] (4) at (.75,.75) {};
            \node[fo] (5) at (0.5,1.5) {};
            \node[fo] (6) at (-0.25, .75) {};
            \node[fo] (7) at (1.25, .75) {};
            \draw (3) -- (1) -- (2) -- (4) -- (5) -- (3);
            \draw (1) -- (6) -- (5);
            \draw (2) -- (7) -- (5);
        \end{tikzpicture}
        \captionof*{figure}{$\varphi_4 = 7$}
    \end{minipage}

    \bigskip 
    
    \begin{minipage}{0.24\linewidth}\centering
        \begin{tikzpicture}[scale=.666*\scf]
            \node[fo] (1) at (-0,0) {};
            \node[fo] (2) at (1,0) {};
            \node[fo] (3) at (0.05,.75) {};
            \node[fo] (4) at (.95,.75) {};
            \node[fo] (5) at (0.5,1.5) {};
            \node[fo] (6) at (-0.25, .75) {};
            \node[fo] (7) at (1.25, .75) {};
            \node[fo] (8) at (0.35, .75) {};
            \node[fo] (9) at (.65, .75) {};
            \node[fo] (10) at (.25, .375) {};
            \node[fo] (11) at (.75, .375) {};
            \draw (3) -- (1) -- (2) -- (4) -- (5) -- (3);
            \draw (1) -- (6) -- (5);
            \draw (2) -- (7) -- (5);
            \draw (5) -- (8) -- (10) -- (11) -- (9) -- (5);
            \draw (10) -- (1) -- (11);
            \draw (10) -- (2) -- (11);
        \end{tikzpicture}
        \captionof*{figure}{$\varphi_5 = 11$}
    \end{minipage}
    \begin{minipage}{0.24\linewidth}\centering
        \begin{tikzpicture}[scale=0.333*\scf]
            \node[fo] (1) at (0,0) {};
            \node[fo] (2) at (0,3) {};
            \node[fo] (3) at (-.375,1) {};
            \node[fo] (4) at (-.375,2) {};
            \node[fo] (5) at (.375,1) {};
            \node[fo] (6) at (.375,2) {};
            \node[fo] (7) at (-1.125,1) {};
            \node[fo] (8) at (-1.125,2) {};
            \node[fo] (9) at (1.125,1) {};
            \node[fo] (10) at (1.125,2) {};
            \draw (1) -- (3) -- (4) -- (2) -- (6) -- (5) -- (1) -- (7) -- (8) -- (2) -- (10) -- (9) -- (1); 
        \end{tikzpicture}
        \captionof*{figure}{$\varphi_6 = 10$}
    \end{minipage}
    \begin{minipage}{0.24\linewidth}\centering
        \begin{tikzpicture}[scale=.666*\scf]
            \node[fo] (1) at (-0,0) {};
            \node[fo] (2) at (1,0) {};
            \node[fo] (3) at (-0.125,.75) {};
            \node[fo] (5) at (0.5,1.5) {};
            \node[fo] (6) at (-0.375, .75) {};
            \node[fo] (7) at (1.375, .75) {};
            \node[fo] (8) at (0.125, .75) {};
            \node[fo] (9) at (.375, .75) {};
            \node[fo] (11) at (.875, .75) {};
            
            \node[fo] (12) at (.25, .375) {};
            \node[fo] (13) at (.75, .375) {};
            
            \draw (3) -- (1) -- (2);
            \draw (5) -- (3);
            \draw (1) -- (6) -- (5);
            \draw (2) -- (7) -- (5);
            \draw (5) -- (8) -- (12);
            \draw (5) -- (9) -- (12);
            \draw (5) -- (11) -- (13);
            \draw (12) -- (13) -- (2);
        \end{tikzpicture}
        \captionof*{figure}{$\varphi_7 = 11$}
    \end{minipage}
    \begin{minipage}{0.24\linewidth}\centering
        \begin{tikzpicture}[scale=.666*\scf]
            \node[fo] (1) at (-0,0) {};
            \node[fo] (2) at (1,0) {};
            \node[fo] (3) at (-0.125,.75) {};
            \node[fo] (4) at (1.125,.75) {};
            \node[fo] (5) at (0.5,1.5) {};
            \node[fo] (6) at (-0.375, .75) {};
            \node[fo] (7) at (1.375, .75) {};
            \node[fo] (8) at (0.125, .75) {};
            \node[fo] (9) at (.375, .75) {};
            \node[fo] (10) at (0.625, .75) {};
            \node[fo] (11) at (.875, .75) {};
            
            \node[fo] (12) at (.25, .375) {};
            \node[fo] (13) at (.75, .375) {};
            
            \draw (3) -- (1) -- (2) -- (4) -- (5) -- (3);
            \draw (1) -- (6) -- (5);
            \draw (2) -- (7) -- (5);
            \draw (5) -- (8) -- (12);
            \draw (5) -- (9) -- (12);
            \draw (5) -- (10) -- (13);
            \draw (5) -- (11) -- (13);
            \draw (1) -- (12) -- (13) -- (2);
            
        \end{tikzpicture}
        \captionof*{figure}{$\varphi_8 = 13$}
    \end{minipage}
    \caption{A smallest graph attaining fault cost $k$ for $k\leq 8$, with the exception of $k = 1$ (we do not know the exact value of $\varphi_1$). See Proposition~\ref{prop:min_order_for_fc}. Each of these graphs is one with largest automorphism group size out of all graphs attaining these values with the exception of $k = 8$, where this is only true for the graphs of girth at least $4$.}
    \label{fig:min_order_for_fc}
    \vspace{\intextsep}
\end{minipage}

\subsection{Fault costs of \texorpdfstring{$\mathbf{3}$}{}-connected graphs}\label{app:counts_3-conn}
\begin{minipage}{\linewidth} \centering
    \vspace{\intextsep}
    \captionsetup{type=table}
    \centering
    \begin{tabular}{c|rrrrr}
         $n\backslash\varphi$ & 0 & 1 & 2 & 3 & 4 \\\hline
         4 & 1&0&0&0&0\\
         5 & 2&0&0&0&0\\
         6 & 9&0&4&0&0\\
         7 & 91&0&20&0&0\\
         8 & 1\,636&0&368&0&0\\
         9 & 58\,119&0&8\,291&0&0\\
         10 & 3\,575\,889&0&326\,455&0&0\\
         11 & 369\,791\,302&0&18\,832\,723&0&81\\
         12 & 63\,452\,885\,511&0&1\,689\,918\,189&0&1\,040\\
         \hline
         13 & 256\,052&0&612\,280&0&330\\
         14 & 11\,309\,365&0&17\,621\,062&215&1\,536\\
         \hline
         15 & 62&0&87&0&0\\
         16 & 984&0&686&0&0\\
         17 & 16\,590&0&7\,292&0&0\\
         18 & 327\,612&0&94\,582&0&0\\
         19 & 7\,213\,982&0&1\,330\,513&0&0\\
         20 & 173\,890\,208&0&21\,401\,341&0&0\\

    \end{tabular}
    \caption{Counts of how many $3$-connected graphs attain each fault cost $\varphi$ for each order $n$. The top part of the table gives counts for all $3$-connected graphs, the middle part for $3$-connected graphs of girth at least $4$, and the bottom part for $3$-connected graphs of girth at least $5$. Fault costs for which the count is zero or which are not included in the table imply that no graphs of the given orders attain this fault cost.}\label{tab:counts_3-conn}
    \vspace{\intextsep}
\end{minipage}

\subsection{Fault costs of \texorpdfstring{$\mathbf{3}$}{3}-connected cubic graphs}\label{app:counts_3-conn_cubic}

\begin{minipage}{\linewidth} \centering
    \vspace{\intextsep}
    \captionsetup{type=table}
    \centering
    \begin{tabular}{c|rrrrr}
         $n\backslash\varphi$ & 0 & 1 & 2\\\hline
         4 & 1&0&0\\
         6 & 1&0&1\\
         8 & 2&0&2\\
         10& 6&0&8\\
         12& 27&0&30\\
         14& 158&0&183\\
         16& 1\,396&0&1\,432\\
         18& 16\,067&0&14\,401\\
         20& 227\,733&0&168\,417\\
         22& 3\,740\,294&0&2\,168\,998\\
         24& 68\,237\,410&0&29\,863\,609\\
         26& 1\,346\,345\,025&0&436\,047\,621\\\hline
         28& 7\,352\,343\,711 & 0 & 1\,103\,915\,037\\\hline
         30& 14\,468\,621\,439&0&152\,588\,083\\
    \end{tabular}
    \caption{Counts of how many $3$-connected cubic graphs attain each fault cost $\varphi$ for each order $n$. The top part of the table gives counts for all $3$-connected graphs cubic , the middle part for $3$-connected cubic graphs of girth at least $4$, and the bottom part for $3$-connected cubic graphs of girth at least $5$. Fault costs for which the count is zero or which are not included in the table imply that no graphs of the given orders attain this fault cost.}\label{tab:counts_3-conn_cubic}
    \vspace{\intextsep}
\end{minipage}

\clearpage

\subsection{Fault costs of \texorpdfstring{$\mathbf{3}$}{3}-connected planar graphs}\label{app:counts_3-conn_planar}

\begin{minipage}{\linewidth} \centering
    \vspace{\intextsep}
    \captionsetup{type=table}
    \centering
    \begin{tabular}{c|rrr}
         $n \backslash \varphi$ & 0 & 1 & 2 \\\hline
         4& 1&0&0\\
         5& 2&0&0\\
         6& 7&0&0\\
         7& 34&0&0\\
         8& 246&0&11\\
         9& 2\,526&0&80\\
         10&30\,842&0&1\,458\\
         11& 416\,108&0&24\,456\\
         12& 5\,955\,716&0&428\,918\\
         13& 88\,766\,610&0&7\,496\,328\\
         14& 1\,364\,787\,597&0&131\,437\,755\\
         15& 21\,522\,536\,388&0&2\,311\,451\,741\\
         16& 346\,748\,111\,059&0&40\,843\,399\,185\\
    \end{tabular}
    \caption{Counts of fault costs for $3$-connected planar graphs. Fault costs for which the count is zero or which are not included in the table imply that no graphs of the given orders attain this fault cost. 
    }
    \label{tab:counts_3-conn_planar}
    \vspace{\intextsep}
\end{minipage}

\clearpage

\end{document}